%% file: main.tex
\documentclass[draftclsnofoot, onecolumn]{IEEEtran}
\usepackage{cite,graphicx,url,amssymb,amsthm,threeparttable,multirow,algorithmic,verbatim,upgreek}
\usepackage[tight,footnotesize]{subfigure}
\usepackage[ruled,vlined]{algorithm2e}
\usepackage[usenames]{color}
\usepackage[normalem]{ulem}
\usepackage[cmex10]{amsmath, mathtools}
\usepackage{caption}
\usepackage{epstopdf}

\IEEEoverridecommandlockouts

\newcommand{\Xm} {\ensuremath{X}}
\newcommand{\Ym} {\ensuremath{y}}
\newcommand{\Strue} {\ensuremath{\mathcal{S}}}

\newcommand{\Struec} {\ensuremath{\mathcal{S}^c}}
\newcommand{\Shat} {\ensuremath{\widehat{\mathcal{S}}_d}}

\newcommand{\numb}{\ensuremath{t}}

\newcommand{\betamin}{\ensuremath{\beta_{\min}}}
\newcommand{\iter}{\ensuremath{l}}

\newcommand{\Sminus}{\ensuremath{{\Strue}^c}}

\newcommand\eqb{\stackrel{\mathclap{\normalfont\mbox{(b)}}}{=}}

\newcommand\eqd{\stackrel{\mathclap{\normalfont\mbox{(d)}}}{=}}

\newcommand\leqc{\stackrel{\mathclap{\normalfont\mbox{(c)}}}{\leq}}
\newcommand\geqa{\stackrel{\mathclap{\normalfont\mbox{(a)}}}{\geq}}

\newcommand{\Gn}{\ensuremath{\mathcal{G}_{\eta}}}
\newcommand{\Gp}{\ensuremath{\mathcal{G}_{p}}}

\newcommand{\ja}{\ensuremath{j_1}}
\newcommand{\jb}{\ensuremath{j_2}}

\newcommand{\Gubar}{\ensuremath{\mathcal{G}_{v}^{'}}}
\DeclarePairedDelimiter\ceil{\lceil}{\rceil}

\mathchardef\mhyphen="2D 

\newtheorem{lemma}{Lemma}
\newtheorem{theorem}{Theorem}
\newtheorem{definition}{Definition}

\newtheorem{corollary}{Corollary}[theorem]

\newcommand{\todel}[1] {}
\newcommand{\toedit}[1]{{#1}}

\begin{document}
\title{ExSIS: Extended Sure Independence Screening for Ultrahigh-dimensional Linear Models}
\author{Talal Ahmed and Waheed U. Bajwa%
\thanks{This work is supported in part by the National Science Foundation under awards CCF-1525276 and CCF-1453073, and by the Army Research Office under award W911NF-14-1-0295. \toedit{Some of the results reported here were presented at IEEE Int. Workshop on Comput. Advances in Multi-Sensor Adaptive Process., 2017~\cite{ahmed2017correlation}}. The authors are with the Department of Electrical and Computer Engineering, Rutgers, The State University of New Jersey, 94 Brett Road, Piscataway, NJ 08854, USA. (Emails: {\tt talal.ahmed@rutgers.edu} and {\tt
waheed.bajwa@rutgers.edu}).}
}

\maketitle
\input{Abstract}

\input{Introduction}
\input{ProblemFormulation}

\input{General_Conditions}

\input{ArbitraryMatrices}
\input{RandomMatrices}
\input{Results}

\input{Conclusion}
\input{Appendix}


\end{document}

%% file: Abstract.tex
\begin{abstract}
Statistical  inference can be computationally prohibitive in ultrahigh-dimensional linear models. Correlation-based variable screening, in which one leverages marginal correlations for removal of irrelevant variables from the model prior to statistical inference, can be used to overcome this challenge. Prior works on correlation-based variable screening either impose statistical priors on the linear model or assume specific post-screening inference methods. This paper first extends the analysis of correlation-based variable screening to arbitrary linear models and post-screening inference techniques. In particular, ($i$) it shows that a condition---termed the screening condition---is sufficient for successful correlation-based screening of linear models, and ($ii$) it provides insights into the dependence of marginal correlation-based screening on different problem parameters. Numerical experiments confirm that these insights are not mere artifacts of analysis; rather, they are reflective of the challenges associated with marginal correlation-based variable screening. Second, the paper explicitly derives the screening condition for \todel{two families of linear models, namely, sub-Gaussian linear models and arbitrary (random or deterministic) linear models}\toedit{arbitrary (random or deterministic) linear models and,} \todel{sub-Gaussian linear models}in the process, it establishes that---under appropriate conditions---it is possible to reduce the dimension of an ultrahigh-dimensional, arbitrary linear model to almost the sample size even when the number of active variables scales almost linearly with the sample size. \toedit{Third, it specializes the screening condition to sub-Gaussian linear models and contrasts the final results to those existing in the literature. This specialization formally validates the claim that the main result of this paper generalizes existing ones on correlation-based screening.} 
\end{abstract}

%% file: Introduction.tex
\section{Introduction}
\label{sec:Introduction}
The ordinary linear model $y = X\beta + \text{noise}$, despite its apparent simplicity, has been the bedrock of signal processing, statistics, and machine learning for decades. The last decade, however, has witnessed a marked transformation of this model: instead of the classical low-dimensional setting in which the dimension, $n$, of $y$ (henceforth, referred to as the sample size) exceeds the dimension, $p$, of $\beta$ (henceforth, referred to as the number of features/predictors/variables), we are increasingly having to operate in the  high-dimensional setting in which the number of variables far exceeds the sample size (i.e., $p \gg n$). While the high-dimensional setting should ordinarily lead to ill-posed problems, the \emph{principle of parsimony}---which states that only a small number of variables typically affect the response $y$---\todel{helps obtain unique solutions to inference problems based on high-dimensional linear models}\toedit{has been employed in the literature to obtain unique solutions for myriad inference problems involving high-dimensional linear models. These inference problems range from sparse regression~\cite{tibshirani1996regression} to sparse representations~\cite{donoho2003optimally} and sparse signal recovery~\cite{candes2008introduction}, with guarantees provided in terms of various properties of $X$ such as \emph{spark}~\cite{donoho2003optimally}, \emph{mutual coherence}~\cite{donoho2003optimally}, \emph{null-space property}~\cite{gribonval2003sparse, cohen2009compressed} \emph{restricted isometry property}~\cite{candes2006robust, candes2008restricted}, etc.}

Our focus in this paper is on \emph{ultrahigh-dimensional} linear models, in which the number of variables can scale exponentially with the sample size: $\log{p} = \mathcal{O}(n^\alpha)$ for $\alpha \in (0,1)$.\footnote{Recall Landau's \emph{big-$\mathcal{O}$} notation: $f(n) = \mathcal{O}(g(n))$ if $\exists C > 0: f(n) \leq C g(n)$ and $f(n) = \Omega(g(n))$ if $g(n) = \mathcal{O}(f(n))$.} Such linear models are increasingly becoming common in application areas ranging from genomics~\cite{golub1999molecular, huang2003linear, helleputte2009partially, stingo2011incorporating} and proteomics~\cite{yasui2003data, clarke2008properties, nesvizhskii2010survey} to sentiment analysis~\cite{pang2008opinion, liu2012sentiment, feldman2013techniques} and hyperspectral imaging~\cite{landgrebe2002hyperspectral, plaza2009recent, bioucas2012hyperspectral}. While there exist a number of techniques in the literature---such as forward selection/matching pursuit, backward elimination~\cite{james2013introduction}, least absolute shrinkage and selection operator (LASSO) \cite{tibshirani1996regression}, elastic net \cite{zou2005regularization}, bridge regression \cite{fu1998penalized, huang2008asymptotic}, adaptive LASSO \cite{zou2006adaptive}, group LASSO \cite{yuan2006model}, and Dantzig selector \cite{candes2007dantzig}---that can be employed for inference from high-dimensional linear models, all these techniques have super-linear (in the number of variables $p$) computational complexity. In the ultrahigh-dimensional setting, therefore, use of the aforementioned methods for statistical inference can easily become computationally prohibitive. Variable selection-based dimensionality reduction, commonly referred to as \emph{variable screening}, has been put forth as a practical means of overcoming this \emph{curse of dimensionality}~\cite{donoho2000high}: since only a small number of (independent) variables actually contribute to the response (dependent variable) in the ultrahigh-dimensional setting, one can first---in principle---discard most of the variables (the screening step) and then carry out inference on a relatively low-dimensional linear model using any one of the sparsity-promoting techniques \toedit{(the inference step). In this work, our focus is on the former step, i.e., the screening step.} There are two main challenges that arise in the context of variable screening in ultrahigh-dimensional linear models. First, the screening algorithm should have low computational complexity (ideally, $\mathcal{O}(np)$). Second, the screening algorithm should be accompanied with mathematical guarantees that ensure the reduced linear model contains \emph{all} relevant variables that affect the response. Our goal in this paper is to revisit one of the simplest screening algorithms, which uses marginal correlations between the variables $\{X_i\}_{i=1}^p$ and the response $y$ for screening purposes~\cite{donoho2006most, genovese2012}, and provide a comprehensive theoretical understanding of its screening performance for arbitrary (random or deterministic) ultrahigh-dimensional linear models.

\subsection{Relationship to Prior Work}
Researchers have long intuited that the (absolute) marginal correlation $|X_i^\top y|$ is a strong indicator of whether the $i$-th variable contributes to the response variable. Indeed, methods such as stepwise forward regression are based on this very intuition. It is only recently, however, that we have obtained a rigorous understanding of the role of marginal correlations in variable screening. One of the earliest screening works in this regard that is agnostic to the choice of the subsequent inference techniques is termed \emph{sure independence screening} (SIS)~\cite{fan2008sure}. SIS is based on simple thresholding of marginal correlations and satisfies the so-called \emph{sure screening} property---which guarantees that all important variables survive the screening stage with high probability---for the case of normally distributed variables. An iterative variant of SIS, termed ISIS, is also discussed in~\cite{fan2008sure}, while~\cite{fan2009ultrahigh} presents variants of SIS and ISIS that can lead to reduced false selection rates of the screening stage. Extensions of SIS to generalized linear models are discussed in~\cite{fan2009ultrahigh,fan2010sure}, while its generalizations for semi-parametric (Cox) models and non-parametric models are presented in~\cite{fan2011nonparametric,fan2014nonparametric} and \cite{fan2010high,zhao2012principled}, respectively.

The marginal correlation $|X_i^\top y|$ can be considered an empirical measure of the Pearson correlation coefficient, which is a natural choice for discovering linear relations between the independent variables and the response. In order to perform ultrahigh-dimensional variable screening in the presence of non-linear relations between $X_i$'s and $y$ and/or heavy-tailed variables,~\cite{hall2009} and~\cite{li2012robust} have put forth screening using generalized (empirical) correlation and Kendall $\tau$ rank correlation, respectively.

The defining characteristics of the works referenced above is that they are agnostic to the inference technique that follows the screening stage. In recent years, screening methods have also been proposed for specific optimization-based inference techniques. To this end,~\cite{ghaoui2010safe} formulates a marginal correlations-based screening method, termed SAFE, for the LASSO problem. \toedit{The goal in SAFE is to discard features that LASSO will not select, and ~\cite{ghaoui2010safe} shows that it} \todel{and shows that SAFE}results in zero false selection rate. In~\cite{tibshirani2012strong}, the so-called \emph{strong rules} for variable screening in LASSO-type problems are proposed that are still based on marginal correlations and that result in discarding of far more variables than the SAFE method. The screening tests of~\cite{ghaoui2010safe,tibshirani2012strong} for the LASSO problem are further improved in \cite{xiang2012fast,dai2012ellipsoid,wang2013lasso} by analyzing the dual of the LASSO problem. We refer the reader to~\cite{xiang2014screening} for an excellent review of these different screening tests for LASSO-type problems.

Notwithstanding these prior works, we have holes in our understanding of variable screening in ultrahigh-dimensional linear models. Works such as~\cite{ghaoui2010safe,tibshirani2012strong,xiang2012fast,dai2012ellipsoid,wang2013lasso} necessitate the use of LASSO-type inference techniques after the screening stage. In addition, these works do not help us understand the relationship between the problem parameters and the dimensions of the reduced model. Stated differently, it is difficult to a priori quantify the computational savings associated with the screening tests proposed in~\cite{ghaoui2010safe,tibshirani2012strong,xiang2012fast,dai2012ellipsoid,wang2013lasso}. Similar to~\cite{fan2008sure,fan2009ultrahigh,hall2009,li2012robust}, and in contrast to~\cite{ghaoui2010safe,tibshirani2012strong,xiang2012fast,dai2012ellipsoid,wang2013lasso}, our focus in this paper is on screening that is agnostic to the post-screening inference technique. To this end,~\cite{hall2009} lacks a rigorous theoretical understanding of variable screening using the generalized correlation. While~\cite{fan2008sure,fan2009ultrahigh,li2012robust} overcome this shortcoming of~\cite{hall2009}, these works have two major limitations. First, their results are derived under the assumption of\todel{restrictive} statistical priors on the linear model (e.g., normally distributed $X_i$'s). In many applications, however, it can be a challenge to ascertain the distribution of the independent variables. Second, the analyses in~\cite{fan2008sure,fan2009ultrahigh,li2012robust} assume the variance of the response variable to be bounded by a constant; this assumption, in turn, imposes the condition $\|\beta\|_2 = \mathcal{O}(1)$. In contrast, defining $\betamin := \min_i |\beta_i|$, we establish in the sequel that the ratio $\frac{\betamin}{\|\beta\|_2}$ (and not $\|\beta\|_2$) directly influences the performance of marginal correlation-based screening procedures.

\subsection{Our Contributions}
Our focus in this paper is on marginal correlation-based screening of ultrahigh-dimensional linear models that is agnostic to the post-screening inference technique. To this end, we provide an extended analysis of the thresholding-based SIS procedure of \cite{fan2008sure}. The resulting screening procedure, which we term \emph{extended sure independence screening} (ExSIS), provides new insights into marginal correlation-based screening of arbitrary (random or deterministic) ultrahigh-dimensional linear models. Specifically, we first provide a simple, distribution-agnostic sufficient condition---termed the \emph{screening condition}---for (marginal correlation-based) screening of linear models. This sufficient condition, which succinctly captures joint interactions among both the active and the inactive variables, is then leveraged to explicitly characterize the performance of ExSIS as a function of various problem parameters, including noise variance, the ratio $\frac{\betamin}{\|\beta\|_2}$, and model sparsity. The numerical experiments reported at the end of this paper confirm that the dependencies highlighted in this screening result are reflective of the actual challenges associated with marginal correlation-based screening and are not mere artifacts of our analysis.

Next, despite the theoretical usefulness of the screening condition, it cannot be explicitly verified in polynomial time for any given linear model. This is reminiscent of related conditions such as the \emph{incoherence condition}~\cite{wainwright2009sharp}, the \emph{irrepresentable condition}~\cite{zhao2006model}, the \emph{restricted isometry property}~\cite{candes2008restricted}, and the \emph{restricted eigenvalue condition}~\cite{raskutti2010restricted} studied in the literature on high-dimensional linear models. In order to overcome this limitation of the screening condition, we explicitly derive it for \todel{two families}\toedit{a family} of linear models \todel{The first family corresponds to sub-Gaussian linear models, in which the independent variables are independently drawn from (possibly different) sub-Gaussian distributions. We show that the ExSIS results for this family of linear models generalize the SIS results derived in~\mbox{\cite{fan2008sure}} for normally distributed linear models. The second first family}\toedit{that} corresponds to arbitrary (random or deterministic) linear models in which the (empirical) correlations between independent variables satisfy certain polynomial-time verifiable conditions. The ExSIS results for this family of linear models establish that, under appropriate conditions, it is possible to reduce the dimension of an ultrahigh-dimensional linear model to almost the sample size even when the number of active variables scales almost linearly with the sample size. This, to the best of our knowledge, is the first screening result that provides such explicit and optimistic guarantees \emph{without} imposing a statistical prior on the distribution of the independent variables. 

\toedit{Finally, in order to highlight the generalizability of our results and their explicit relationship to existing results}, we also derive the screening condition for a family of sub-Gaussian linear models in which the independent variables are independently drawn from (possibly different) sub-Gaussian distributions. We show that the ExSIS results for this family of linear models \toedit{match}\todel{coincide} with the SIS results derived in~\cite{fan2008sure} for the case of normally distributed linear models. \toedit{This validates}\todel{, and thus they corroborate} the generality of the screening condition in relation to existing works.

\subsection{Notation and Organization}
The following notation is used throughout this paper. Lower-case letters are used to denote scalars and vectors, while upper-case letters are used to denote matrices. Given $a \in \mathbb{R}$, $\lceil a\rceil$ denotes the smallest integer greater than or equal to $a$. Given $q \in \mathbb{Z}_+$, we use $[[q]]$ as a shorthand for $\{1,\dots,q\}$. Given a vector $v$, $\|v\|_p$ denotes its $\ell_p$ norm. Given a matrix $A$, $A_j$ denotes its $j$-th column and $A_{i,j}$ denotes the entry in its $i$-th row and $j$-th column. Further, given a set $\mathcal{I} \subset \mathbb{Z}_+$, $A_{\mathcal{I}}$ (resp., $v_{\mathcal{I}}$) denotes a submatrix (resp., subvector) obtained by retaining columns of $A$ (resp., entries of $v$) corresponding to the indices in $\mathcal{I}$. Finally, the superscript $(\cdot)^{\top}$ denotes the transpose operation.

The rest of this paper is organized as follows. We formulate the problem of marginal correlation-based screening in Sec.~\ref{sec:ProblemFormulation}. Next, in Sec.~\ref{sec:GeneralConditions}, we define the screening condition and present one of our main results that establishes the screening condition as a sufficient condition for ExSIS. \todel{In Sec.~\ref{sec:RandomMatrices}, we derive the screening condition for sub-Gaussian linear models and discuss the resulting ExSIS guarantees in relation to prior work.} In Sec.~\ref{sec:ArbitraryMatrices}, we derive the screening condition for arbitrary linear models based on the correlations between independent variables and discuss implications of the derived ExSIS results. \toedit{In Sec.~\ref{sec:RandomMatrices}, we derive the screening condition for sub-Gaussian linear models and discuss the resulting ExSIS guarantees in relation to prior work.} Finally, results of extensive numerical experiments on both synthetic and real data are reported in Sec.~\ref{sec:Results}, while concluding remarks are presented in Sec.~\ref{sec:Conclusion}. 

%% file: ProblemFormulation.tex
\section{Problem Formulation}
\label{sec:ProblemFormulation}
Our focus in this paper is on the ultrahigh-dimensional ordinary linear model $y = X \beta + \eta$, where $y\in \mathbb{R}^n$, $X \in \mathbb{R}^{n \times p}$, and $\log{p} = \mathcal{O}(n^\alpha)$ for $\alpha \in (0,1)$. In the \todel{statistics}literature, $X$ is referred to as data/design/observation/sensing matrix with the rows of $X$ corresponding to individual observations and the columns of $X$ corresponding to individual features/predictors/variables, $y$ is referred to as \toedit{measurement}/observation/response vector with individual responses given by $\{y_i\}_{i=1}^{n}$, $\beta$ is referred to as the parameter vector, and $\eta$ is referred to as modeling error or \toedit{measurement}/observation noise. Throughout this paper, we assume $X$ has unit $\ell_2\mhyphen$norm columns, $\beta \in \mathbb{R}^{p}$ is $k\mhyphen$sparse with $k < n$ (i.e., $\big|\{i \in [[p]]:\beta_i \not= 0\}\big| = k < n$), and $\eta \in \mathbb{R}^p$ is a zero-mean Gaussian vector with (entry-wise) variance $\sigma^2$ and covariance $C_\eta = \sigma^2 I$. Here, $\eta$ is taken to be Gaussian with covariance $\sigma^2 I$ for the sake of this exposition, but our analysis is trivially generalizable to other noise distributions and/or covariance matrices. Further, we make no a priori assumption on the distribution of $X$. Finally, we define $\Strue := \{i \in [[p]]:\beta_i \not= 0\}$ to be the set that indexes the non-zero components of $\beta$. Using this notation, the linear model can equivalently be expressed as
\begin{align}
\Ym = \Xm \beta + \eta = X_{\Strue} \beta_{\Strue} + \eta \text{.}
\label{eq:sys}
\end{align}

Given \eqref{eq:sys}, the goal of variable screening is to reduce the number of variables in the linear model from $p$ (since $p \ggg n$) to a moderate scale $d$ (with $d \lll p$) using a fast and efficient method. Our focus here is in particular on screening methods that satisfy the so-called \emph{sure screening} property~\cite{fan2008sure}; specifically, a method is said to carry out sure screening if the $d\mhyphen$dimensional model returned by it is guaranteed with high probability to retain all the columns of $X$ that are indexed by $\Strue$. The motivation here is that once one obtains a moderate-dimensional model through sure screening of \eqref{eq:sys}, one can use computationally intensive model selection, regression and estimation techniques on the $d\mhyphen$dimensional model for reliable model selection (identification of $\Strue$), prediction (estimation of $X \beta$), and reconstruction (estimation of $\beta$), respectively.

In this paper, we study sure screening using marginal correlations between the response vector and the columns of $X$. The resulting screening procedure is outlined in Algorithm~\ref{algo:correlation_screening}, which is based on the principle that the higher the correlation of a column of $X$ with the response vector, the more likely it is that the said column contributes to the response vector (i.e., it is indexed by the set $\Strue$).
\begin{algorithm}[!h]
\algsetup{indent=1em}
\begin{algorithmic}[1]
\STATE \textbf{Input:} Design matrix $X \in \mathbb{R}^{n\times p}$, response vector $y\in \mathbb{R}^{n}$, and $d\in\mathbb{Z}_{+}$
\STATE \textbf{Output:} $\Shat \subset [[p]]$ such that $|\Shat|=d$
\STATE $w \gets X^\top y$
\STATE $\Shat \gets \{ i \in [[p]] : |w_i| \text{ is among the $d$ largest of all marginal correlations} \}$
\caption{Marginal Correlation-based Independence Screening}
\label{algo:correlation_screening}
\end{algorithmic}
\end{algorithm}

The computational complexity of Algorithm~\ref{algo:correlation_screening} is only $\mathcal{O}(np)$ and its ability to screen ultrahigh-dimensional linear models has been investigated in recent years by a number of researchers~\cite{donoho2006most, genovese2012}. The fundamental difference among these works stems from the manner in which the parameter $d$ (the dimension of the screened model) is computed from \eqref{eq:sys}. Our goal in this paper is to provide an extended understanding of the screening performance of Algorithm~\ref{algo:correlation_screening} for \emph{arbitrary} (random or deterministic) design matrices. The term \emph{sure independence screening} (SIS) was coined in~\cite{fan2008sure} to refer to screening of ultrahigh-dimensional \emph{sub-Gaussian} linear models using Algorithm~\ref{algo:correlation_screening}.\footnote{\toedit{The analysis in \cite{fan2008sure} requires a certain concentration property that has only been shown in \cite{fan2008sure} to hold for Gaussian matrices. However, this property can also be proved for sub-Gaussian matrices by appealing to matrix concentration inequalities; see, e.g., \cite{rudelson2009smallest}.}} In this vein, we refer to variable screening using Algorithm~\ref{algo:correlation_screening} and the analysis of this paper as \emph{extended sure independence screening} (ExSIS). The main research challenge for ExSIS is specification of $d$ for arbitrary matrices such that $\Strue \subset \Shat$ with high probability. Note that there is an inherent trade-off in addressing this challenge: the higher the value of $d$, the more likely is $\Shat$ to satisfy the sure screening property; however, the smaller the value of $d$, the lower the computational cost of performing model selection, regression, estimation, etc., on the $d\mhyphen$dimensional problem. This leads us to the following research questions for ExSIS: ($i$) What are the conditions on $X$ under which $\Strue \subset \Shat$? ($ii$) How small can $d$ be for arbitrary matrices such that $\Strue \subset \Shat$? ($iii$) What are the constraints on the sparsity parameter $k$ under which $\Strue \subset \Shat$? Note that there is also an interplay between the sparsity level $k$ and the allowable value of $d$ for sure screening: the lower the sparsity level, the easier it should be to screen a larger number of columns of $X$. Thus, an understanding of ExSIS also requires characterization of this relationship between $k$ and $d$ for marginal correlation-based screening. In the sequel, we not only address the aforementioned questions for ExSIS, but also characterize this relationship.

%% file: General_Conditions.tex
\section{Sufficient Conditions for Sure Screening}
\label{sec:GeneralConditions}
In this section, we derive the most general sufficient conditions for ExSIS of ultrahigh-dimensional linear models. The results reported in this section provide important insights into the workings of ExSIS \emph{without} imposing any statistical priors on $X$ and $\beta$. We begin with a definition of the \emph{screening condition} for the design matrix $X$.
\begin{definition}[$(k, b)\mhyphen$Screening Condition]
Fix an arbitrary $\beta \in \mathbb{R}^{p}$ that is $k\mhyphen$sparse. The (normalized) matrix $X$ satisfies the $(k, b)\mhyphen$screening condition if there exists $0<b(n,p)<\frac{1}{\sqrt{k}}$ such that the following hold:
\begin{align*}
\max \limits_{i\in \Strue} | \sum \limits_{\mathclap{\substack{j\in \Strue \\ j \neq i}}} X_i^{\top} \Xm_{j} \beta_j | \leq b(n,p) \|\beta \|_2 \text{,} \tag{SC-1}\label{eqn:SC-1}
\nonumber\\
\max \limits_{i\in \Sminus} |\sum \limits_{j\in \Strue} X_i^{\top} \Xm_{j} \beta_j | \leq b(n,p) \|\beta \|_2 \text{.} \tag{SC-2}\label{eqn:SC-2}
\end{align*}
\label{def:sc1}
\end{definition}
\noindent The screening condition is a statement about the collinearity of the independent variables in the design matrix. The parameter $b(n,p)$ in the screening condition captures the similarity between ($i$) the columns of $X_{\Strue}$, and ($ii$) the columns of $X_{\Strue}$ and $X_{\Struec}$; the smaller the parameter $b(n,p)$ is, the less similar the columns are. Furthermore, since $k<(b(n,p))^{-2}$ in the screening condition, the parameter $b(n,p)$ reflects constraints on the sparsity parameter $k$.

We now present one of our main screening results for arbitrary design matrices, which highlights the significance of the screening condition and the role of the parameter $b(n,p)$ within ExSIS.
\begin{theorem}[Sufficient Conditions for ExSIS]
Let $y=X\beta+\eta$ with $\beta$ a $k\mhyphen$sparse vector and the entries of $\eta$ independently distributed as $\mathcal{N}(0,\sigma^2)$. Define $\betamin := \min \limits_{i\in\Strue} |\beta_i|$ and $\tilde{\eta} := X^\top \eta$, and let $\Gn$ be the event $\{\|\tilde{\eta}\|_{\infty} \leq 2 \sqrt{\sigma^2 \log p} \}$. Suppose $X$ satisfies the screening condition and assume $\frac{\betamin}{\|\beta\|_2} > 2 b(n,p) + 4 \frac{\sqrt{\sigma^2 \log p}}{\|\beta\|_2}$. Then, conditioned on $\Gn$, Algorithm 1 satisfies $\Strue \subset \Shat$ as long as $d \geq \ceil*{\frac{\sqrt{k}}{\frac{\betamin}{\|\beta\|_2} - 2 b(n,p) - \frac{4 \sqrt{\sigma^2 \log p}}{\|\beta\|_2}}}$.
\label{th:main_result_gen}
\end{theorem}
We refer the reader to Sec.~\ref{ssec:proof.thm.mainresult} for a proof of this theorem.

\subsection{Discussion}\label{ssec:gen.cond.disc}
Theorem~\ref{th:main_result_gen} highlights the dependence of ExSIS on the observation noise, the ratio $\frac{\betamin}{\|\beta\|_2}$, the parameter $b(n,p)$, and model sparsity. We first comment on the relationship between ExSIS and observation noise $\eta$. Notice that the statement of Theorem \ref{th:main_result_gen} is dependent upon the event $\Gn$. However, for any $\epsilon>0$, we have (see, e.g., \cite[Lemma~6]{bajwa2010gabor})
\begin{align}
\Pr(\|\tilde{\eta}\|_{\infty} \geq \sigma \epsilon) < \frac{4 p}{\epsilon \sqrt{2 \pi}} \exp \Big( - \frac{\epsilon^2}{2} \Big)\text{.}
\label{eq:eta}
\end{align}
Therefore, substituting $\epsilon = 2 \sqrt{\log p}$ in \eqref{eq:eta}, we obtain
\begin{align}
\Pr( \Gn ) \geq 1 - 2 (p \sqrt{2 \pi \log p})^{-1} \text{.}
\label{eq:eta2.final}
\end{align}
Thus, Algorithm~\ref{algo:correlation_screening} possesses the sure screening property in the case of the observation noise $\eta$ distributed as $\mathcal{N}(0,\sigma^2 I)$. We further note from the statement of Theorem~\ref{th:main_result_gen} that the higher the \emph{signal-to-noise ratio} (SNR), defined here as $\text{SNR} := \frac{\|\beta\|_2}{\sigma}$, the more Algorithm~\ref{algo:correlation_screening} can screen irrelevant/inactive variables. It is also worth noting here trivial generalizations of Theorem~\ref{th:main_result_gen} for other noise distributions. In the case of $\eta$ distributed as $\mathcal{N}(0,C_\eta)$, Theorem~\ref{th:main_result_gen} has $\sigma^2$ replaced by the largest eigenvalue of the covariance matrix $C_\eta$. In the case of $\eta$ following a non-Gaussian distribution, Theorem~\ref{th:main_result_gen} has $2 \sqrt{\sigma^2 \log p}$ replaced by distribution-specific upper bound on $\|X^\top \eta\|_{\infty}$ that holds with high probability.

In addition to the noise distribution, the performance of ExSIS also seems to be impacted by the \emph{minimum-to-signal ratio} (MSR), defined here as $\text{MSR} := \frac{\betamin}{\|\beta\|_2} \in \big(0,\frac{1}{\sqrt{k}}\big]$. Specifically, the higher the MSR, the more Algorithm~\ref{algo:correlation_screening} can screen inactive variables. Stated differently, the independent variable with the weakest contribution to the response determines the size of the screened model. Finally, the parameter $b(n,p)$ in the screening condition also plays a central role in characterization of the performance of ExSIS. First, the smaller the parameter $b(n,p)$, the more Algorithm~\ref{algo:correlation_screening} can screen inactive variables. Second, the smaller the parameter $b(n,p)$, the more independent variables can be active in the original model; indeed, we have from the screening condition that $k<(b(n,p))^{-2}$. Third, the smaller the parameter $b(n,p)$, the lower the smallest allowable value of MSR; indeed, we have from the theorem statement that $\text{MSR} > 2 b(n,p) + 4 \frac{\sqrt{\sigma^2 \log p}}{\|\beta\|_2}$.

It is evident from the preceding discussion that the screening condition (equivalently, the parameter $b(n,p)$) is one of the most important factors that helps understand the workings of ExSIS and helps quantify its performance. Unfortunately, the usefulness of this knowledge is limited in the sense that the screening condition cannot be utilized in practice. Specifically, the screening condition is defined in terms of the set $\Strue$, which is of course unknown. We overcome this limitation of Theorem~\ref{th:main_result_gen} by implicitly deriving the screening condition for \todel{sub-Gaussian design matrices in Sec.~\mbox{\ref{sec:RandomMatrices}} and for}a class of arbitrary (random or deterministic) design matrices in Sec.~\ref{sec:ArbitraryMatrices} and sub-Gaussian design matrices in Sec.~\ref{sec:RandomMatrices}.

\subsection{Proof of Theorem~\ref{th:main_result_gen}}\label{ssec:proof.thm.mainresult}

We first provide an outline of the proof of Theorem \ref{th:main_result_gen}, which is followed by its formal proof. Define $p_0 := p$, $\widehat{\mathcal{S}}_{p_0} := [[p]]$, and $\numb_1 := | \{ j \in \widehat{\mathcal{S}}_{p_0}: |w_j| \geq \underset{i \in \Strue} \min | w_i |  \} |$. Next, fix a positive integer $p_1 < p_0$ and define
$$\widehat{\mathcal{S}}_{p_1} := \{ i \in \widehat{\mathcal{S}}_{p_0} : |w_i| \text{ is among the $p_1$ largest of all marginal correlations} \} \text{.}$$
The idea is to first derive an initial upper bound on $\numb_1$, denoted by $\bar{\numb}_1$, and then choose $p_1=\ceil*{\bar{\numb}_1}$; trivially, we have $\Strue \subset \widehat{\mathcal{S}}_{p_1} \subset \widehat{\mathcal{S}}_{p_0}$. As a result, we get
\begin{align}
y = X \beta + \eta
 = X_{\widehat{\mathcal{S}}_{p_0}} \beta_{\widehat{\mathcal{S}}_{p_0}} + \eta
 = X_{\widehat{\mathcal{S}}_{p_1}} \beta_{\widehat{\mathcal{S}}_{p_1}} + \eta \text{.}
\label{eq:sysp1}
\end{align}
Note that while deriving $\bar{\numb}_1$, we need to ensure $\bar{\numb}_1 < p_0$; this in turn imposes some conditions on $X$ that also need to be specified. Next, we can repeat the aforementioned steps to obtain $\widehat{\mathcal{S}}_{p_2}$ from $\widehat{\mathcal{S}}_{p_1}$ for a fixed positive integer $p_2 < p_1 < p_0$. Specifically, define
\begin{align*}
\widehat{\mathcal{S}}_{p_2} := \{ i \in \widehat{\mathcal{S}}_{p_1} : |w_i| \text{ is among the $p_2$ largest of all marginal correlations} \}
\end{align*}
and $\numb_2 := | \{ j \in \widehat{\mathcal{S}}_{p_1} : |w_j| \geq \underset{i \in \Strue} \min | w_i |  \} |$.
We can then derive an upper bound on $\numb_2$, denoted by $\bar{\numb}_2$, and then choose $p_2=\ceil*{\bar{\numb}_2}$; once again, we have $\Strue \subset \widehat{\mathcal{S}}_{p_2}\subset \widehat{\mathcal{S}}_{p_1}\subset \widehat{\mathcal{S}}_{p_0}$. Notice further that we do require $\bar{\numb}_2 < p_1$, which again will impose conditions on $X$.

In similar vein, we can keep on repeating this procedure to obtain a decreasing sequence of numbers $\{\bar{\numb}_j\}_{j=1}^i$ and sets $\widehat{\mathcal{S}}_{p_0} \supset \widehat{\mathcal{S}}_{p_1} \supset \widehat{\mathcal{S}}_{p_2} \supset \ldots \supset \widehat{\mathcal{S}}_{p_i} \supset \Strue$ as long as $\bar{\numb}_i < p_{i-1}$, where $\big\{p_j :=  \ceil*{\bar{\numb}_j}\big\}_{j=1}^i$ and $i \in \mathbb{Z}_{+}$. The complete proof of Theorem \ref{th:main_result_gen} follows from a careful combination of these (analytical) steps. In order for us to be able to do that, however, we need two lemmas. The first lemma provides an upper bound on $\numb_i = | \{ j \in \widehat{\mathcal{S}}_{p_{i-1}}: |w_j| \geq \underset{i \in \Strue} \min | w_i |  \} |$ for $i\in\mathbb{Z}_{+}$, denoted by $\bar{\numb}_i$. The second lemma provides conditions on the design matrix $X$ such that $\bar{\numb}_i < p_{i-1}$. The proof of the theorem follows from repeated application of the two lemmas.

\begin{lemma}
Fix $i\in\mathbb{Z}_{+}$ and suppose $\Strue \subset \widehat{\mathcal{S}}_{p_{i-1}}$, where $|\widehat{\mathcal{S}}_{p_{i-1}}| =: p_{i-1}$ and $p_{i-1} \leq p$. Further, suppose the design matrix $X$ satisfies the $(k, b)\mhyphen$screening condition for the $k\mhyphen$sparse vector $\beta$ and the event $\Gn$ holds true. Finally, define $\numb_i := | \{ j\in \widehat{\mathcal{S}}_{p_{i-1}}: |w_j| \geq \underset{i \in \Strue} \min | w_i |  \} |$. Under these conditions, we have
\begin{align}
\numb_i \leq \frac{ p_{i-1} b(n,p) \|\beta\|_2 + \|\beta\|_1 + 2 p_{i-1} \sqrt{\sigma^2 \log p}}
 {\betamin - b(n,p) \|\beta\|_2  - 2 \sqrt{\sigma^2 \log p} } =: \bar{\numb}_i \text{.}
 \label{eq:frac0_lemma}
\end{align}
\label{lemma:frac0}
\end{lemma}
The proof of this lemma is provided in Appendix~\ref{app:proof_GeneralConditions.lemma.frac01}. The second lemma, whose proof is given in Appendix~\ref{app:proof_GeneralConditions.one.iter.stoc}, provides conditions on $X$ under which the upper bound derived on $\numb_i$ for $i\in\mathbb{Z}_{+}$, denoted by $\bar{\numb}_i$, is non-trivial.
\begin{lemma}
Fix $i\in\mathbb{Z}_{+}$. Suppose $p_{i-1} > \frac{\sqrt{k}}{\frac{\betamin}{\|\beta\|_2} - 2 b(n,p) - \frac{4 \sqrt{\sigma^2 \log p}}{\|\beta\|_2}}$ and $\frac{\betamin}{\|\beta\|_2} > 2 b(n,p) + \frac{4 \sqrt{\sigma^2 \log p}}{\|\beta\|_2}$. Then, we have $0<\bar{\numb}_i < p_{i-1}$.
\label{lemma:gen_one_iter_stoc}
\end{lemma}
\noindent We are now ready to present a complete technical proof of Theorem~\ref{th:main_result_gen}.
\begin{proof}
The idea is to use Lemma~\ref{lemma:frac0} and Lemma~\ref{lemma:gen_one_iter_stoc} \emph{repeatedly} to screen columns of $X$. Note, however, that this is simply an analytical technique and we do not \emph{actually} need to perform such an iterative procedure to specify $d$ in Algorithm~\ref{algo:correlation_screening}. To begin, recall that we have $p_0 := p$, $\widehat{\mathcal{S}}_{p_0} := [[p]]$, $$\widehat{\mathcal{S}}_{p_1} := \{ i \in \widehat{\mathcal{S}}_{p_0} : |w_i| \text{ is among the $p_1$ largest of all marginal correlations}\},$$ and $p_1 = \ceil*{\bar{\numb}_1}$, where $\bar{\numb}_1$ is defined in \eqref{eq:frac0_lemma}. By Lemma~\ref{lemma:frac0} and Lemma~\ref{lemma:gen_one_iter_stoc}, we have $\Strue \subset \widehat{\mathcal{S}}_{p_1}$ and $p_1<p_0$, respectively. Next, given $p_1 > \ceil*{\frac{\sqrt{k}}{\frac{\betamin}{\|\beta\|_2} - 2 b(n,p) - \frac{4 \sqrt{\sigma^2 \log p}}{\|\beta\|_2}}}$, we can use Lemma \ref{lemma:frac0} and Lemma \ref{lemma:gen_one_iter_stoc} to obtain $\widehat{\mathcal{S}}_{p_2}$ from $\widehat{\mathcal{S}}_{p_1}$ in a similar fashion. Specifically, let
$$\widehat{\mathcal{S}}_{p_2} := \{ i \in \widehat{\mathcal{S}}_{p_1} : |w_i| \text{ is among the $p_2$ largest of all marginal correlations}\}$$ and $p_2 = \ceil*{\bar{\numb}_2}$, where $\bar{\numb}_2$ is defined in \eqref{eq:frac0_lemma}. Then, by Lemma \ref{lemma:frac0} and Lemma \ref{lemma:gen_one_iter_stoc}, we have $\Strue \subset \widehat{\mathcal{S}}_{p_2}$ and $p_2<p_1$, respectively.

Notice that we can keep on repeating this procedure to obtain sub-models $\widehat{\mathcal{S}}_{p_1}, \widehat{\mathcal{S}}_{p_2}, \ldots, \widehat{\mathcal{S}}_{p_\iter}$ such that $p_{\iter} \leq \frac{\sqrt{k}}{\frac{\betamin}{\|\beta\|_2} - 2 b(n,p) - \frac{4 \sqrt{\sigma^2 \log p}}{\|\beta\|_2}}$ and $p_{\iter - 1} > \frac{\sqrt{k}}{\frac{\betamin}{\|\beta\|_2} - 2 b(n,p) - \frac{4 \sqrt{\sigma^2 \log p}}{\|\beta\|_2}}$. By repeated applications of Lemma~\ref{lemma:frac0} and Lemma~\ref{lemma:gen_one_iter_stoc}, we have $\Strue \subset \widehat{\mathcal{S}}_{p_\iter}$. Further, we are also guaranteed that $p_{\iter} \leq \frac{\sqrt{k}}{\frac{\betamin}{\|\beta\|_2} - 2 b(n,p) - \frac{4 \sqrt{\sigma^2 \log p}}{\|\beta\|_2}}$. Thus, we can choose $d \geq \ceil*{\frac{\sqrt{k}}{\frac{\betamin}{\|\beta\|_2} - 2 b(n,p) - \frac{4 \sqrt{\sigma^2 \log p}}{\|\beta\|_2}}}$ in Algorithm~\ref{algo:correlation_screening} in one shot and have $\Strue \subset \Shat$.
\end{proof}

%% file: ArbitraryMatrices.tex
\section{Screening of Arbitrary Design Matrices}
\label{sec:ArbitraryMatrices}

\todel{The ExSIS analysis in Sec.~\ref{sec:RandomMatrices} specializes Theorem~\ref{th:main_result_gen} for sub-Gaussian design matrices. But what about the design matrices in which either the entries do not follow sub-Gaussian distributions or the statistical distributions of entries are unknown? We address this particular question in this section by deriving verifiable sufficient conditions that guarantee the screening condition for any arbitrary (random or deterministic) design matrix.}

\toedit{In this section, we specialize Theorem~\ref{th:main_result_gen} for a family of arbitrary linear models in which the design matrices satisfy certain polynomial-time verifiable conditions. We achieve this by deriving verifiable sufficient conditions that guarantee the screening condition for any arbitrary (random or deterministic) design matrix.} These sufficient conditions are presented in terms of two measures of similarity among the columns of a design matrix. These measures, termed \emph{worst-case coherence} and \emph{average coherence}, are defined as follows.
\begin{definition}[Worst-case and Average Coherences]
Let $X$ be an $n \times p$ matrix with unit $\ell_2$-norm columns. The worst-case coherence of $X$ is denoted by $\mu$ and is defined as~\cite{davis1997adaptive}: $$\mu := \underset{i,j:i\neq j} \max \big|X_i^\top X_j\big|.$$ On the other hand, the average coherence of $X$ is denoted by $\nu$ and is defined as~\cite{bajwa2010gabor}: $$\nu := \frac{1}{p-1} \underset{i} \max \bigg| \underset{j:j\neq i} \sum X_i^\top X_j \bigg|.$$
\end{definition}
\noindent Notice that both the worst-case and the average coherences are readily computable in polynomial time. Heuristically, the worst-case coherence is an indirect measure of pairwise similarity among the columns of $X$: $\mu \in [0,1]$ with $\mu \searrow 0$ as the columns of $X$ become less similar and $\mu \nearrow 1$ as at least two columns of $X$ become more similar. The average coherence, on the other hand, is an indirect measure of both the collective similarity among the columns of $X$ and the spread of the columns of $X$ within the unit sphere: $\nu \in [0,\mu]$ with $\nu \searrow 0$ as the columns of $X$ become more spread out in $\mathbb{R}^n$ and $\nu \nearrow \mu$ as the columns of $X$ become less spread out. We refer the reader to \cite{bajwa2012two} for further discussion of these two measures as well as their values for commonly encountered matrices.

We are now ready to describe the main results of this section. The first result connects the screening condition to the worst-case coherence. We will see, however, that this result suffers from the so-called square-root bottleneck: ExSIS analysis based solely on the worst-case coherence can, at best, handle $k = \mathcal{O}(\sqrt{n})$ scaling of the sparsity parameter. The second result overcomes this bottleneck by connecting the screening condition to both worst-case and average coherences. The caveat here is that this result imposes a mild statistical prior on the set $\Strue$.

\subsection{ExSIS and the Worst-case Coherence}
We begin by relating the worst-case coherence of an arbitrary design matrix $X$ with unit-norm columns to the screening condition.
\begin{lemma}[Worst-case Coherence and the Screening Condition]
Let $\Xm$ be an $n \times p$ design matrix with unit-norm columns. Then, we have
\begin{align*}
\max \limits_{i\in \Strue} | \sum \limits_{\mathclap{\substack{j\in \Strue \\ j \neq i}}} X_i^{\top} \Xm_{j} \beta_j |
&\leq \mu \sqrt{k} \|\beta\|_2 \textnormal{, and}\\
\max \limits_{i\in \Struec} |\sum \limits_{j\in \Strue} X_i^{\top} \Xm_{j} \beta_j |
&\leq \mu \sqrt{k} \|\beta\|_2 \text{.}
\end{align*}
\label{lemma:Det_mu}
\end{lemma}
\noindent The proof of this lemma is provided in Appendix~\ref{app:proof_ArbitraryMatrices.mu}. It follows from Lemma~\ref{lemma:Det_mu} that a design matrix satisfies the screening condition with parameter $b(n,p) = \mu \sqrt{k}$ as long as $\mu < k^{-1} \Leftrightarrow k < \mu^{-1}$. We now combine this implication of Lemma~\ref{lemma:Det_mu} with Theorem~\ref{th:main_result_gen} to provide a result for ExSIS of arbitrary linear models.
\begin{theorem}
Let $y=X\beta+\eta$ with $\beta$ a $k\mhyphen$sparse vector and the entries of $\eta$ independently distributed as $\mathcal{N}(0,\sigma^2)$. Suppose $k < \mu^{-1}$ and $\frac{\betamin}{\|\beta\|_2} > 2 \mu \sqrt{k} + 4 \frac{\sqrt{\sigma^2 \log p}}{\|\beta\|_2}$. Then, Algorithm~\ref{algo:correlation_screening} satisfies $\Strue \subset \Shat$ with probability exceeding $1 - 2 (p \sqrt{2 \pi \log p})^{-1}$ as long as $d \geq \ceil*{\frac{\sqrt{k}}{\frac{\betamin}{\|\beta\|_2} - 2 \mu \sqrt{k} - \frac{4 \sqrt{\sigma^2 \log p}}{\|\beta\|_2}}}$.
\label{th:mu_conditions}
\end{theorem}

The proof of this theorem follows directly from Lemma~\ref{lemma:Det_mu} and Theorem~\ref{th:main_result_gen}. Next, a straightforward corollary of Theorem~\ref{th:mu_conditions} shows that ExSIS of arbitrary linear models can in fact be carried out without explicit knowledge of the sparsity parameter $k$.
\begin{corollary}
Let $y=X\beta+\eta$ with $\beta$ a $k\mhyphen$sparse vector and the entries of $\eta$ independently distributed as $\mathcal{N}(0,\sigma^2)$. Suppose $p\geq2n$, $k < \mu^{-1}$, and $\frac{\betamin}{\|\beta\|_2} > 2 c_1 \mu \sqrt{k} + 4 c_2 \frac{\sqrt{\sigma^2 \log p}}{\|\beta\|_2}$ for some $c_1, c_2>2$. Then, Algorithm~\ref{algo:correlation_screening} satisfies $\Strue \subset \Shat$ with probability exceeding $1 - 2 (p \sqrt{2 \pi \log p})^{-1}$ as long as $d\geq \ceil*{\sqrt{n} \, }$.
\label{corr:mu_conditions}
\end{corollary}
\begin{proof}
Under the assumption of $\frac{\betamin}{\|\beta\|_2} > 2 c_1 \mu \sqrt{k} + \frac{4 c_2 \sqrt{\sigma^2 \log p}}{\|\beta\|_2}$, notice that
\begin{align}
d \geq \ceil*{\frac{\sqrt{k}}{ 2(c_1 -1) \mu \sqrt{k} + \frac{4 (c_2 -1) \sqrt{\sigma^2 \log p}}{\|\beta\|_2}}}
\label{eq:Det_d1}
\end{align}
is a sufficient condition for $d \geq \ceil*{\frac{\sqrt{k}}{\frac{\betamin}{\|\beta\|_2} - 2 \mu \sqrt{k} - \frac{4 \sqrt{\sigma^2 \log p}}{\|\beta\|_2}}}$. Further, note that $d \geq \ceil*{ (2 \mu)^{-1} }$ is a sufficient condition for \eqref{eq:Det_d1}. Next, since $p\geq 2n$, we also have $\mu^{-1} \leq \sqrt{2n}$ from the Welch bound on the worst-case coherence of design matrices~\cite{welch1974lower}. Thus, $d\geq \ceil*{\sqrt{n}\,}$ is a sufficient condition for $d \geq \ceil*{\frac{\sqrt{k}}{\frac{\betamin}{\|\beta\|_2} - 2 \mu \sqrt{k} - \frac{4 \sqrt{\sigma^2 \log p}}{\|\beta\|_2}}}$.
\end{proof}

\toedit{It is worth reflecting here on the fact that Corollary~\ref{corr:mu_conditions} allows the number of remaining variables in the screened model to scale as $d\geq \ceil*{\sqrt{n}\,}$. While this means that the screened model can have as few as $\ceil*{\sqrt{n}\,}$ variables, the result suffers from the so-called square-root bottleneck as it constraints sparsity to scale only as $\mathcal{O}(\sqrt{n}\,)$: $k = \mathcal{O}(\mu^{-1}) \Rightarrow k = \mathcal{O}(\sqrt{n}\,)$ because of the Welch bound. We overcome this sparsity limitation by adding the average coherence into the mix and imposing a uniform statistical prior on the true model $\Strue$ in the next section.}

\todel{It is interesting to compare this result for arbitrary linear models with Corollary~\ref{corr:main_result_sub_Gaussian} for sub-Gaussian linear models. Corollary~\ref{corr:main_result_sub_Gaussian} requires the size of the screened model to scale as $\mathcal{O}(n/\log{p})$, whereas this result requires $d$ to scale only as $\mathcal{O}(\sqrt{n})$. While this may seem to suggest that Corollary~\ref{corr:mu_conditions} is better than Corollary~\ref{corr:main_result_sub_Gaussian}, such an observation ignores the respective constraints on the sparsity parameter $k$ in the two results. Specifically, Corollary~\ref{corr:main_result_sub_Gaussian} allows for almost linear scaling of the sparsity parameter, $k = \mathcal{O}(n/\log{p})$, whereas Corollary~\ref{corr:mu_conditions} suffers from the so-called square-root bottleneck: $k = \mathcal{O}(\mu^{-1}) \Rightarrow k = \mathcal{O}(\sqrt{n})$ because of the Welch bound. Stated differently, Corollary~\ref{corr:mu_conditions} fails to specialize to Corollary~\ref{corr:main_result_sub_Gaussian} for the case of $X$ being a sub-Gaussian design matrix. We overcome this limitation of the results of this section by adding the average coherence into the mix and imposing a statistical prior on the true model $\Strue$ in the next section.}

\subsection{ExSIS and the Coherence Property}\label{ssec:exsis_cohProperty}
In order to break the square-root bottleneck for ExSIS of arbitrary linear models, we first define the notion of the coherence property.
\begin{definition}[The Coherence Property]
An $n \times p$ design matrix $\Xm$ with unit-norm columns is said to obey the coherence property if there exists a constant $c_{\mu}>0$ such that $\mu < \frac{1}{c_{\mu }\sqrt{\log p}}$ and $\nu < \frac{\mu}{\sqrt{n}}$.
\label{def:coherence}
\end{definition}
\noindent Heuristically, the coherence property, which was first introduced in~\cite{bajwa2010gabor}, requires the independent variables to be sufficiently (marginally and jointly) uncorrelated. Notice that, unlike many conditions in high-dimensional statistics (see, e.g., \cite{wainwright2009sharp,zhao2006model,candes2008restricted,raskutti2010restricted}), the coherence property is explicitly certifiable in polynomial time for any given design matrix. We now establish that the coherence property implies the design matrix satisfies the screening condition with high probability, where the probability is with respect to uniform prior on the true model $\Strue$.
\begin{lemma}[Coherence Property and the Screening Condition]
Let $\Xm$ be an $n \times p$ design matrix that satisfies the coherence property with $c_{\mu}>10\sqrt{2}$, and suppose $p\geq \max \{2n, \exp(5)\}$ and $k\leq \frac{n}{\log p}$. Further, assume $\Strue$ is drawn uniformly at random from $k$-subsets of $[[p]]$. Then, with probability exceeding $1 - 4 p^{-1}$, we have
\begin{align*}
 &\max \limits_{i\in \Strue} | \sum \limits_{\mathclap{\substack{j\in \Strue \\ j \neq i}}} X_i^{\top} \Xm_{j} \beta_j | \leq c_{\mu} \mu \sqrt{\log p} \|\beta\|_2, \quad \textnormal{and}
 \\
 &\max \limits_{i\in \Struec} | \sum \limits_{j\in \Strue} X_i^{\top} \Xm_{j} \beta_j | \leq c_{\mu} \mu \sqrt{\log p} \|\beta\|_2.
\end{align*}
\label{lemma:StOC_SC}
\end{lemma}

The proof of this lemma is provided in Appendix~\ref{app:proof_ArbitraryMatrices.mu.nu}. Lemma~\ref{lemma:StOC_SC} implies that a design matrix that obeys the coherence property also satisfies the screening condition for \emph{most} models with $b(n,p) = c_{\mu} \mu \sqrt{\log p}$ as long as $\mu < c_\mu^{-1} (k \log{p})^{-1/2} \Leftrightarrow k < \frac{\mu^{-2}}{c_\mu^2 \log{p}}$. Comparing this with Lemma~\ref{lemma:Det_mu} and the resulting constraint $k < \mu^{-1}$ for the screening condition to hold in the case of arbitrary design matrices, we see that---at the expense of uniform prior on the true model---the coherence property results in a better bound on the screening parameter as long as $\log{p} = \mathcal{O}(\mu^{-1})$. We can now utilize Lemma~\ref{lemma:StOC_SC} along with Theorem~\ref{th:main_result_gen} to provide an improved result for ExSIS of arbitrary linear models.
\begin{theorem}
Let $y=X\beta+\eta$ with $\beta$ a $k\mhyphen$sparse vector and the entries of $\eta$ independently distributed as $\mathcal{N}(0,\sigma^2)$. Further, assume $\Xm$ satisfies the coherence property with $c_{\mu}>10\sqrt{2}$ and $\Strue$ is drawn uniformly at random from $k$-subsets of $[[p]]$. Finally, suppose $p\geq \max \{2n, \exp(5)\}$, $k < \frac{\mu^{-2}}{c_\mu^2 \log{p}}$, and $\frac{\betamin}{\|\beta\|_2} > 2 c_{\mu} \mu \sqrt{\log p} + \frac{4 \sqrt{\sigma^2 \log p}}{\|\beta\|_2}$. Then, Algorithm~\ref{algo:correlation_screening} satisfies $\Strue \subset \Shat$ with probability exceeding $1 - 6p^{-1}$ as long as $d \geq \ceil*{\frac{\sqrt{k}}{\frac{\betamin}{\|\beta\|_2} - 2 c_{\mu} \mu \sqrt{\log p} - \frac{4 \sqrt{\sigma^2 \log p}}{\|\beta\|_2}}}$.
\label{th:mu_nu_conditions}
\end{theorem}
The proof of this theorem is omitted here since it follows in a straightforward manner from Lemma~\ref{lemma:StOC_SC}, Theorem~\ref{th:main_result_gen}, and a union bound argument. Nonetheless, it is worth mentioning here that the $k\leq \frac{n}{\log p}$ bound in Lemma~\ref{lemma:StOC_SC} is omitted in Theorem~\ref{th:mu_nu_conditions} since $k < \frac{\mu^{-2}}{c_\mu^2 \log{p}} \Rightarrow k\leq \frac{n}{\log p}$ because of the Welch bound. The final result of this section is a corollary of Theorem~\ref{th:mu_nu_conditions} that removes the dependence of $d$ on knowledge of the problem parameters.
\begin{corollary}
Let $y=X\beta+\eta$ with $\beta$ a $k\mhyphen$sparse vector and the entries of $\eta$ independently distributed as $\mathcal{N}(0,\sigma^2)$. Further, assume $\Xm$ satisfies the coherence property with $c_{\mu}>10\sqrt{2}$ and $\Strue$ is drawn uniformly at random from $k$-subsets of $[[p]]$. Finally, suppose $p\geq \max \{2n, \exp(5)\}$, $k < \frac{\mu^{-2}}{c_\mu^2 \log{p}}$, and $\frac{\betamin}{\|\beta\|_2} > 2 c_{\mu} c_1 \mu \sqrt{ \log p} + 4 c_2 \frac{\sqrt{\sigma^2 \log p}}{\|\beta\|_2}$ for some $c_1, c_2>2$. Then, Algorithm~\ref{algo:correlation_screening} satisfies $\Strue \subset \Shat$ with probability exceeding $1 - 6p^{-1}$ as long as $d \geq \ceil*{\frac{n}{\log p }}$.
\label{corr:mu_nu_conditions}
\end{corollary}
\begin{proof}
Since $\frac{\betamin}{\|\beta\|_2} > 2 c_1 c_{\mu} \mu \sqrt{\log p} + \frac{4 c_2 \sqrt{\sigma^2 \log p}}{\|\beta\|_2}$, we have that
\begin{align}
d \geq \ceil*{\frac{\sqrt{k}}{ 2(c_1 -1) c_{\mu} \mu \sqrt{\log p} + \frac{4 (c_2-1) \sqrt{\sigma^2 \log p}}{\|\beta\|_2} }}
\label{eq:Det2_d1}
\end{align}
is a sufficient condition for $d \geq \ceil*{\frac{\sqrt{k}}{\frac{\betamin}{\|\beta\|_2} - 2 c_{\mu} \mu \sqrt{\log p} - \frac{4 \sqrt{\sigma^2 \log p}}{\|\beta\|_2}}}$. The claim now follows because $d \geq \ceil*{\frac{n}{\log p }}$ is a sufficient condition for \eqref{eq:Det2_d1}, owing to the facts that $p\geq 2n$ and the Welch bound imply $\mu^{-1} \leq \sqrt{2n}$ and $k < \frac{\mu^{-2}}{c_\mu^2 \log{p}} \Rightarrow k\leq \frac{n}{\log p}$.
\end{proof}

\subsection{Discussion}
\label{ssec:ArbitraryDiscussion}
Both Theorem~\ref{th:mu_conditions} and Theorem~\ref{th:mu_nu_conditions} shed light on the feasibility of marginal correlation-based screening of linear models \emph{without} imposing a statistical prior on the design matrix. While Theorem~\ref{th:mu_conditions} in this regard provides the least restrictive results, it does suffer from the square-root bottleneck: $k = \mathcal{O}(\mu^{-1}) \Rightarrow k = \mathcal{O}(\sqrt{n})$. Theorem~\ref{th:mu_nu_conditions}, on the other hand, overcomes this bottleneck at the expense of uniform prior on the true model as long as $\log{p} = \mathcal{O}(\mu^{-1})$; in this case, the condition on the sparsity parameter becomes $k = \mathcal{O}\left(\mu^{-2}/\log{p}\right)$. Therefore, Theorem~\ref{th:mu_nu_conditions} allows for sparsity scaling as high as $k = \mathcal{O}(n/\log{p})$ for design matrices with $\mu = \mathcal{O}(n^{-1/2})$; see~\cite{bajwa2012two} for existence of such matrices. In addition, Theorem~\ref{th:mu_conditions} and Theorem~\ref{th:mu_nu_conditions} also differ from each other in terms of their respective constraints on $\frac{\betamin}{\|\beta\|_2}$ for feasibility of marginal correlation-based screening; the constraint in Theorem~\ref{th:mu_nu_conditions} is less restrictive than in Theorem~\ref{th:mu_conditions} for $\log{p} = \mathcal{O}(\mu^{-1})$.

\toedit{Given that the constraints on sparsity and $\frac{\betamin}{\|\beta\|_2}$ in Theorem~\ref{th:mu_nu_conditions} are less restrictive than in Theorem~\ref{th:mu_conditions}, it is instructive to specialize the results in Theorem~\ref{th:mu_nu_conditions} for screening of some specific family of linear models. For this purpose,}\todel{A natural question to ask at this point is whether Theorem~\ref{th:mu_nu_conditions} specializes to Theorem~\ref{th:main_result_sub_Gaussian}. The answer to this question is in the affirmative, except for some small penalties that one has to pay because of the fact that Theorem~\ref{th:mu_nu_conditions} does not exploit any sub-Gaussianity of the entries of $\Xm$ in its analysis. In order to illustrate this further,}
we consider the case of Gaussian design matrices and reproduce bounds on their worst-case and average coherences from~\cite{bajwa2012two}.
\begin{lemma}[\!\!\mbox{\cite[Theorem~8]{bajwa2012two}}]
Let $V = [V_{i,j}]$ be an $n \times p$ matrix with the entries $\{V_{i,j}\}_{i,j=1}^{n,p}$ independently distributed as $\mathcal{N}(0,1)$ and $60\log{p} \leq n \leq \frac{p-1}{4\log{p}}$. Suppose the design matrix $X$ is obtained by normalizing the columns of $V$, i.e., $X = V \text{diag}(1/\|V_1\|_2,\dots,1/\|V_p\|_2)$. Then, with probability exceeding $1 -  11p^{-1}$, we have
\begin{align*}
  \mu &\leq \frac{\sqrt{15 \log{p}}}{\sqrt{n}-\sqrt{12\log{p}}}, \quad \textnormal{and}\\
  \nu &\leq \frac{\sqrt{15 \log{p}}}{n-\sqrt{12n\log{p}}}.
\end{align*}
\label{lemma:Gaussian_Coherence}
\end{lemma}
It can be seen from this lemma that Gaussian design matrices satisfy the coherence property for $\log{p} = \mathcal{O}(n^{1/2})$. We can therefore use \toedit{Lemma~\ref{lemma:Gaussian_Coherence} to specialize Corollary~\ref{corr:mu_nu_conditions} to Gaussian matrices and obtain the following result.  
\begin{corollary}
Let $V = [V_{i,j}]$ be an $n \times p$ matrix with the entries $\{V_{i,j}\}_{i,j=1}^{n,p}$ independently distributed as $\mathcal{N}(0,1)$, the design matrix $X$ be obtained by normalizing the columns of $V$, and $y=X\beta+\eta$ with $\beta$ a $k\mhyphen$sparse vector and the entries of $\eta$ independently distributed as $\mathcal{N}(0,\sigma^2)$. Further, assume $\Strue$ is drawn uniformly at random from $k$-subsets of $[[p]]$. Finally, suppose $p\geq \max \{2n, \exp(5)\}$, $n\geq 2$, $\log{p}<\frac{\sqrt{n}}{c_{\mu}\sqrt{15}}$, $k < \frac{n}{15 c_\mu^2 (\log{p})^2}$, and $\frac{\betamin}{\|\beta\|_2} > 2 \sqrt{15} c_{\mu} c_1 \frac{ \log{p}}{\sqrt{n}} + 4 c_2 \frac{\sqrt{\sigma^2 \log p}}{\|\beta\|_2}$ for some $c_1, c_2>2, c_{\mu}>10\sqrt{2}$. Then, Algorithm~\ref{algo:correlation_screening} satisfies $\Strue \subset \Shat$ with probability exceeding $1 - 17p^{-1}$ as long as $d \geq \ceil*{\frac{n}{\log p }}$.
\label{corr:mu_nu_Gaussian_conditions}
\end{corollary}
From Corollary~\ref{corr:mu_nu_Gaussian_conditions}, we can} conclude that screening of Gaussian linear models using Algorithm~\ref{algo:correlation_screening} can be carried out with $d \geq \ceil*{\frac{n}{\log p }}$ as long as: ($i$) $\log{p} = \mathcal{O}(n^{1/2})$, ($ii$) $k = \mathcal{O}(n/(\log{p})^2)$, and ($iii$) $\frac{\betamin}{\|\beta\|_2} = \Omega\left(\frac{\log{p}}{\sqrt{n}} + \frac{\sqrt{\sigma^2 \log p}}{\|\beta\|_2}\right)$. \toedit{We now compare this specialization of Theorem~\ref{th:mu_nu_conditions} for Gaussian matrices with the known results of SIS in the existing literature. To this end, we focus on the results reported in~\cite{fan2008sure}, which is one of the most influential SIS works for the family of Gaussian linear models.} \todel{In contrast to the screening condition presented in this paper, the analysis in \mbox{\cite{fan2008sure}} is carried out for design matrices that satisfy a certain concentration property. Since the said concentration property has only been shown in \mbox{\cite{fan2008sure}} to hold for Gaussian matrices, our discussion in the following is limited to Gaussian design matrices with independent entries.}\toedit{Specifically, we assume $\sigma = 0$ for the sake of simplicity of argument, and explicitly compare Corollary~\ref{corr:mu_nu_Gaussian_conditions} with \cite[Theorem~1]{fan2008sure} for the case of Gaussian design matrices with independent entries. The SIS results reported in \cite[Theorem~1]{fan2008sure} hold under four specific conditions. In particular, Condition~3 in \cite{fan2008sure} requires that: ($i$) the variance of the response variable is $\mathcal{O}(1)$, ($ii$) $\betamin \geq \frac{c_\kappa}{n^\kappa}$ for some $c_{\kappa} >0$, $\kappa \geq 0$, and ($iii$) $\min_{i\in\Strue}|\text{cov}(\beta_i^{-1}Y,X_i)| \geq c_3$ for some $c_3 > 0$. Notice that the $\mathcal{O}(1)$ variance condition is equivalent to having $\|\beta\|_2 = \mathcal{O}(1)$. We therefore also impose the condition $\|\beta\|_2 = \mathcal{O}(1)$ in Corollary~\ref{corr:mu_nu_Gaussian_conditions} for comparison purposes. 

Comparing the two results, we see that~\cite{fan2008sure} ensures sure screening with slightly less-restrictive constraints. In~\cite[Theorem~1]{fan2008sure}, SIS requires $\betamin = \Omega(n^{-\kappa})$ for $\kappa < 1/2$ and $\log p = \mathcal{O}(n^\alpha)$ for some $\alpha \in (0, 1 - 2\kappa)$. In contrast, substituting $\betamin = \Omega(n^{-\kappa})$ and $\log p = \mathcal{O}(n^\alpha)$ in Corollary~\ref{corr:mu_nu_Gaussian_conditions} leads to constraints of $\kappa < 1/2$ and $0<\alpha<\frac{1}{2}-\kappa$. The impact of this small gap in $\alpha$ for the two results can be more concretely analyzed in terms of the respective sparsity constraints. Specifically, \cite[Theorem~1]{fan2008sure} imposes the sparsity constraint $k=\mathcal{O}(n^{2 \kappa})$ for the sure screening result to hold. Invoking, the condition $\log p = \mathcal{O}(n^\alpha)$ with $\alpha \in (0, 1 - 2 \kappa)$ reduces this constraint to $k=\mathcal{O}(n^{1-\alpha})=\mathcal{O}\left(\frac{n}{\log p}\right)$, which is slightly less restrictive than the sparsity constraint of $k = \mathcal{O}(n/(\log{p})^2)$ imposed in Corollary~\ref{corr:mu_nu_Gaussian_conditions}.} 

\toedit{The aforementioned slight gaps in constraints lead to the following questions: are these gaps because Corollary~\ref{corr:mu_nu_Gaussian_conditions} is based on the coherence property, which is only a sufficient condition for the original condition; or are these gaps because the screening condition is fundamentally weaker than the conditions derived in~\cite[Theorem~1]{fan2008sure} for sure screening? In the next section, we show that the answer to the first question is in the affirmative by evaluating the screening condition explicitly for sub-Gaussian design matrices and then specializing Theorem~\ref{th:main_result_gen} to obtain tight conditions for sure screening of such matrices.}

\todel{To summarize, the ExSIS results derived in this paper coincide with the ones in \mbox{\cite{fan2008sure}} for the case of Gaussian design matrices. However, our results are more general in the sense that they explicitly bring out the dependence of Algorithm~\ref{algo:correlation_screening} on the SNR and the MSR, which is something missing in \mbox{\cite{fan2008sure}}.} \todel{, and they are applicable to sub-Gaussian design matrices.}

\todel{Comparing this with Corollary~\ref{corr:main_result_sub_Gaussian} in general and the discussion in Sec.~\ref{ssec:rand.mat.disc} in particular, we see that the general theory of Sec.~\ref{ssec:exsis_cohProperty} \emph{almost} matches with the specialized theory of Sec.~\ref{sec:RandomMatrices}. Specifically, compared to the constraints of $\log{p} = \mathcal{O}(n^{1/2})$, $k = \mathcal{O}(n/(\log{p})^2)$, and $\frac{\betamin}{\|\beta\|_2} = \Omega\left(\frac{\log{p}}{\sqrt{n}} + \frac{\sqrt{\sigma^2 \log p}}{\|\beta\|_2}\right)$ arising in Sec.~\ref{ssec:exsis_cohProperty} for Gaussian design matrices, Sec.~\ref{sec:RandomMatrices} results in slightly less restrictive constraints of $\log{p} = \mathcal{O}(n)$, $k = \mathcal{O}(n/\log{p})$, and $\frac{\betamin}{\|\beta\|_2} = \Omega\left(\sqrt{\frac{\log{p}}{n}} + \frac{\sqrt{\sigma^2 \log p}}{\|\beta\|_2}\right)$. These small gaps are the price one has to pay for the generality of Theorem~\ref{th:mu_nu_conditions}.}

%% file: RandomMatrices.tex
\section{Screening of Sub-Gaussian Design Matrices}
\label{sec:RandomMatrices}

In this section, we \toedit{explicitly} characterize the implications of Theorem~\ref{th:main_result_gen} for ExSIS of the family of sub-Gaussian design matrices. As noted in Sec.~\ref{sec:GeneralConditions}, this effort primarily involves establishing the screening condition for sub-Gaussian matrices and specifying the parameter $b(n,p)$ for such matrices. We begin by first recalling the definition of a sub-Gaussian random variable.
\begin{definition}
A zero-mean random variable $\mathcal{X}$ is said to follow a sub-Gaussian distribution $subG(B)$ if there exists a sub-Gaussian parameter $B>0$ such that $\mathbb{E}[\exp (\lambda \mathcal{X})] \leq \exp \Big(\frac{B^2 \lambda^2}{2}\Big)$ for all $\lambda\in\mathbb{R}$.
\end{definition}
Our focus in this \todel{paper}\toedit{section} is on design matrices in which entries are first independently drawn from sub-Gaussian distributions and then the columns are normalized. In contrast to prior works, however, we do not require the (pre-normalized) entries to be identically distributed. Rather, we allow each independent variable to be distributed as a sub-Gaussian random variable with a different sub-Gaussian parameter. Thus, the ExSIS analysis of this section is applicable to design matrices in which different columns might have different sub-Gaussian distributions. It is also straightforward to extend our analysis to the case where all (and not just across column) entries of the design matrix are non-identically distributed; we do not focus on this extension in here for the sake of notational clarity.

\subsection{Main Result}
The ExSIS of linear models involving sub-Gaussian design matrices mainly requires establishing the screening condition and characterization of the parameter $b(n,p)$ for sub-Gaussian matrices. We accomplish this by individually deriving~\eqref{eqn:SC-1} and~\eqref{eqn:SC-2} in Definition~\ref{def:sc1} for sub-Gaussian design matrices in the following two lemmas.
\begin{lemma}
Let $V = [V_{i,j}]$ be an $n \times p$ matrix with the entries $\{V_{i,j}\}_{i,j=1}^{n,p}$ independently distributed as $subG(b_j)$ with variances $\mathbb{E}[V_{i,j}^2] = \sigma_j^2$. Suppose the design matrix $X$ is obtained by normalizing the columns of $V$, i.e., $X = V \text{diag}(1/\|V_1\|_2,\dots,1/\|V_p\|_2)$. Finally, fix an arbitrary $\beta \in \mathbb{R}^{p}$ that is $k\mhyphen$sparse, define $\frac{\sigma_*}{b_*} := \min \limits_{j\in\Strue} \frac{\sigma_j}{b_j}$, and let $\log p \leq \frac{n}{16} (\frac{\sigma_*}{4 b_*})^4$. Then, with probability exceeding $1 -  \frac{2k^2}{p^2}$, we have
\begin{align*}
\max \limits_{i\in \Strue} | \sum \limits_{\mathclap{\substack{j\in \Strue \\ j \neq i}}} X_i^{\top} \Xm_{j} \beta_j |
\leq \sqrt{\frac{8 \log p}{n}} \Big(\frac{b_*}{\sigma_*} \Big)\,\|\beta\|_2.
\end{align*}
\label{lemma:subGa}
\end{lemma}
\begin{lemma}
Let $V = [V_{i,j}]$ be an $n \times p$ matrix with the entries $\{V_{i,j}\}_{i,j=1}^{n,p}$ independently distributed as $subG(b_j)$ with variances $\mathbb{E}[V_{i,j}^2] = \sigma_j^2$. Suppose the design matrix $X$ is obtained by normalizing the columns of $V$, i.e., $X = V \text{diag}(1/\|V_1\|_2,\dots,1/\|V_p\|_2)$. Finally, fix an arbitrary $\beta \in \mathbb{R}^{p}$ that is $k\mhyphen$sparse, define $\frac{\sigma_*}{b_*} := \min \limits_{j\in\Strue} \frac{\sigma_j}{b_j}$, and let $\log p \leq \frac{n}{16} (\frac{\sigma_*}{4 b_*})^4$. Then, with probability exceeding $1 -  \frac{2(k+1)(p-k)}{p^2}$, we have
\begin{align*}
\max \limits_{i\in \Struec} | \sum \limits_{\mathclap{\substack{j\in \Strue}}} X_i^{\top} \Xm_{j} \beta_j |
\leq \sqrt{\frac{8 \log p}{n}} \Big(\frac{b_*}{\sigma_*} \Big)\,\|\beta\|_2.
\end{align*}
\label{lemma:subGb}
\end{lemma}
The proofs of Lemma~\ref{lemma:subGa} and Lemma~\ref{lemma:subGb} are provided in Appendix~\ref{app:proof_RandomMatrices.subGa} and Appendix~\ref{app:proof_RandomMatrices.subGb}, respectively. It now follows from a simple union bound argument that the screening condition holds for sub-Gaussian design matrices with probability exceeding $1 - 2 (k+1) p^{-1}$. In particular, we have from Lemma~\ref{lemma:subGa} and Lemma~\ref{lemma:subGb} that $b(n,p) = \sqrt{\frac{8 \log p}{n} } (\frac{b_*}{\sigma_*})$ for sub-Gaussian matrices. We can now use this knowledge and Theorem~\ref{th:main_result_gen} to provide the main result for ExSIS of ultrahigh-dimensional linear models involving sub-Gaussian design matrices.
\begin{theorem}[ExSIS and Sub-Gaussian Matrices]
Let $V = [V_{i,j}]$ be an $n \times p$ matrix with the entries $\{V_{i,j}\}_{i,j=1}^{n,p}$ independently distributed as $subG(b_j)$ with variances $\mathbb{E}[V_{i,j}^2] = \sigma_j^2$. Suppose the design matrix $X$ is obtained by normalizing the columns of $V$, i.e., $X = V \text{diag}(1/\|V_1\|_2,\dots,1/\|V_p\|_2)$. Next, let $y=X\beta+\eta$ with $\beta$ a $k\mhyphen$sparse vector and the entries of $\eta$ independently distributed as $\mathcal{N}(0,\sigma^2)$. Finally, define $\frac{\sigma_*}{b_*} := \min \limits_{j\in\Strue} \frac{\sigma_j}{b_j}$ and $\betamin := \min \limits_{i\in\Strue} |\beta_i|$, and let $\log p \leq \frac{n}{16} (\frac{\sigma_*}{4 b_*})^4$ and $\frac{\betamin}{\|\beta\|_2} > 2 \sqrt{\frac{8 \log p}{n}}(\frac{b_*}{\sigma_*}) + 4 \frac{\sqrt{\sigma^2 \log p}}{\|\beta\|_2}$. Then Algorithm~\ref{algo:correlation_screening} guarantees $\Strue \subset \Shat$ with probability exceeding $1-2 (k+2) p^{-1}$ as long as
\begin{align*}
d \geq \ceil*{\frac{\sqrt{k}}{\frac{\betamin}{\|\beta\|_2} - 2 \sqrt{\frac{8 \log p}{n}} \Big(\frac{b_*}{\sigma_*} \Big) - \frac{4 \sqrt{\sigma^2 \log p}}{\|\beta\|_2}}} \text{.}
\end{align*}
\label{th:main_result_sub_Gaussian}
\end{theorem}
\begin{proof}
Let $\Gp$ be the event that the design matrix $X$ satisfies the screening condition with parameter $b(n,p) = \sqrt{\frac{8 \log p}{n} } (\frac{b_*}{\sigma_*})$. Further, let $\Gn$ be the event as defined in Theorem~\ref{th:main_result_gen}. It then follows from Lemma~\ref{lemma:subGa}, Lemma~\ref{lemma:subGb}, \eqref{eq:eta2.final}, and the union bound that the event $\Gp \cap \Gn$ holds with probability exceeding $1 - 2 (k+2) p^{-1}$. The advertised claim now follows directly from Theorem~\ref{th:main_result_gen}.
\end{proof}

\subsection{Discussion}\label{ssec:rand.mat.disc}
Since Theorem~\ref{th:main_result_sub_Gaussian} follows from Theorem~\ref{th:main_result_gen}, it shares many of the insights discussed in Sec.~\ref{ssec:gen.cond.disc}. In particular, Theorem~\ref{th:main_result_sub_Gaussian} allows for exponential scaling of the number of independent variables, $\log p \leq \frac{n}{16} (\frac{\sigma_*}{4 b_*})^4$, and dictates that the number of independent variables, $d$, retained after the screening stage be increased with an increase in the sparsity level and/or the number of independent variables, while it can be decreased with an increase in the SNR, MSR, and/or the number of samples. Notice that the lower bound on $d$ in Theorem~\ref{th:main_result_sub_Gaussian} does require knowledge of the sparsity level. However, this limitation can be overcome in a straightforward manner, as shown below.
\begin{corollary}
Let $V = [V_{i,j}]$ be an $n \times p$ matrix with the entries $\{V_{i,j}\}_{i,j=1}^{n,p}$ independently distributed as $subG(b_j^2)$ with variances $\mathbb{E}[V_{i,j}^2] = \sigma_j^2$. Suppose the design matrix $X$ is obtained by normalizing the columns of $V$, i.e., $X = V \text{diag}(1/\|V_1\|_2,\dots,1/\|V_p\|_2)$. Next, let $y=X\beta+\eta$ with $\beta$ a $k\mhyphen$sparse vector and the entries of $\eta$ independently distributed as $\mathcal{N}(0,\sigma^2)$. Further, define $\frac{\sigma_*}{b_*} := \min \limits_{j\in\Strue} \frac{\sigma_j}{b_j}$ and $\betamin := \min \limits_{i\in\Strue} |\beta_i|$. Finally, let $k \leq \frac{n}{\log p}$, $\log p \leq \frac{n}{16} (\frac{\sigma_*}{4 b_*})^4$, and $\frac{\betamin}{\|\beta\|_2} > 2 c_1 \sqrt{\frac{8 \log p}{n}}(\frac{b_*}{\sigma_*}) + 4 c_2 \frac{\sqrt{\sigma^2 \log p}}{\|\beta\|_2}$ for some constants $c_1, c_2>2$. Then Algorithm~\ref{algo:correlation_screening} guarantees $\Strue \subset \Shat$ with probability exceeding $1-2 (k+2) p^{-1}$ as long as $d \geq \ceil*{\frac{n}{\log p}}$.
\label{corr:main_result_sub_Gaussian}
\end{corollary}
\begin{proof}
Theorem~\ref{th:main_result_sub_Gaussian} and the condition $\frac{\betamin}{\|\beta\|_2} > 2 c_1 \sqrt{\frac{8 \log p}{n}} \Big(\frac{b_*}{\sigma_*} \Big) + \frac{4 c_2 \sqrt{\sigma^2 \log p}}{\|\beta\|_2}$ dictates
\begin{align}
d \geq \ceil*{\frac{\sqrt{k}}{ 2(c_1 -1) \sqrt{\frac{8 \log p}{n}} \Big(\frac{b_*}{\sigma_*} \Big) + \frac{4 (c_2 -1) \sqrt{\sigma^2 \log p}}{\|\beta\|_2}}}
\label{eq:G_d2}
\end{align}
for sure screening of sub-Gaussian design matrices. The claim now follows by noting that $d \geq \ceil*{\frac{n}{\log p}}$ is a sufficient condition for \eqref{eq:G_d2} since $k \leq \frac{n}{\log p}$ and $\sigma_* \leq b_*$ for sub-Gaussian random variables.
\end{proof}
\toedit{Following our discussion in Sec.~\ref{ssec:ArbitraryDiscussion}, a few remarks are in order now concerning the comparison between Corollary~\ref{corr:main_result_sub_Gaussian} and~\cite[Theorem~1]{fan2008sure}. We compare the two results for the case of Gaussian matrices with independent entries by once again imposing $\sigma=0$ and $\|\beta\|_2 = \mathcal{O}(1)$. Recall that \cite[Theorem~1]{fan2008sure} imposes constraints $\betamin = \Omega(n^{-\kappa})$ for $\kappa < 1/2$ and $\log p = \mathcal{O}(n^\alpha)$ for some $\alpha \in (0, 1 - 2\kappa)$. In comparison, substituting $\betamin=\Omega(n^{-\kappa})$ and $\log p = \mathcal{O}(n^\alpha)$ in Corollary~\ref{corr:main_result_sub_Gaussian} results in identical constraints of $\kappa < 1/2$ and $\alpha<1-2\kappa$. Further, the sparsity constraint $k = \mathcal{O}(n/ \log{p})$ in~\cite[Theorem~1]{fan2008sure} also matches the sparsity constraint imposed in Corollary~\ref{corr:main_result_sub_Gaussian}.} To summarize, the ExSIS results derived in this section coincide with the ones in \cite{fan2008sure} for the case of Gaussian design matrices. However, our results are more general in the sense that they explicitly bring out the dependence of Algorithm~\ref{algo:correlation_screening} on the SNR and the MSR, which is something missing in \cite{fan2008sure}.

%% file: Results.tex
\section{Experimental Results}
\label{sec:Results}

In this section, we present results from two synthetic-data experiments and one real-data experiment. In Section \ref{sec:Results_Oracle}, we demonstrate that our insights on the ExSIS procedure are not mere artifacts of our analysis. In Section \ref{sec:Results_Gaussian}, we analyze the performance of various regularization-based screening procedures in comparison to the ExSIS procedure. Finally, in Section \ref{sec:Results_IMDb}, we analyze the computational savings achieved from the use of ExSIS for screening of the feature space as part of sentiment analysis of IMDb movie reviews~\cite{maas-EtAl:2011}.

\subsection{Oracle-based Screening}
\label{sec:Results_Oracle}

In order to ensure the insights offered by Theorem~\ref{th:main_result_gen} are not mere artifacts of our analysis, we carry out numerical experiments to study the impact of relevant parameters on the screening performance of an \emph{oracle} that has perfect knowledge of the minimum value of $d$ required in Algorithm~\ref{algo:correlation_screening} to ensure $\Strue \subset \Shat$. In particular, we use these oracle-based experiments to verify the role of $b(n,p)$ and MSR in screening using Algorithm~\ref{algo:correlation_screening}, as specified by Theorem \ref{th:main_result_gen}. Since worst-case coherence is an indirect measure of pairwise similarity among the columns of $X$, and since $b(n,p)$ cannot be explicitly computed, we use $\mu$ as a surrogate for the value of $b(n,p)$ in our experiments.

The design matrix $\Xm\in \mathbb{R}^{n \times p}$ in our experiments is generated such that it consists of independent and identically distributed Gaussian entries, followed by normalization of the columns of $X$. Among other parameters, $n=500$, $p=2000$, $k=5$, and $\sigma =0$ in the experiments. The entries of $\Strue$ are chosen uniformly at random from $[[p]]$. Furthermore, the non-zero entries in the parameter vector $\beta$ are sampled from a uniform distribution $U[a,e]$; the value of $a$ is set at $1$ whereas $e\in[2,10]$. Finally, the experiments comprise the use of an oracle to find the minimum possible value of $d$ that can be used in Algorithm~\ref{algo:correlation_screening} while ensuring $\Strue \subset \Shat$. We refer to this minimum value of $d$ as the \emph{minimum model size} (MMS), and we use median of MMS over $400$ runs of the experiment as a metric of difficulty of screening.

To analyze the impact of increasing $\mu$ (equivalently, $b(n,p)$) and MSR on screening using Algorithm~\ref{algo:correlation_screening}, the numerical experiments are repeated for various values of $\mu$ and MSR. In particular, the worst-case coherence of $\Xm$ is varied by scaling its largest singular value, followed by normalization of the columns of $X$, while the MSR is increased by decreasing the value of $e$. In Fig.~\ref{fig:d_mu_oracle_median}, we plot the median MMS against $\mu$ for different MSR values. The experimental results of the oracle performance offer two interesting insights. First, the median MMS increases with $\mu$; this shows that any analysis for screening using Algorithm~\ref{algo:correlation_screening} needs to account for the similarity between the columns of $\Xm$. This relationship is captured by the parameter $b(n,p)$ in Theorem~\ref{th:main_result_gen}. Second, the difficulty of screening for an oracle increases with decreasing MSR values. This relationship is also reflected in Theorem~\ref{th:main_result_gen}: as $\| \beta \|_2$ increases for a fixed $e$, MSR decreases and the median MMS increases.

\begin{figure}[t!]
\centering
\subfigure[]{ \includegraphics[scale=0.45]{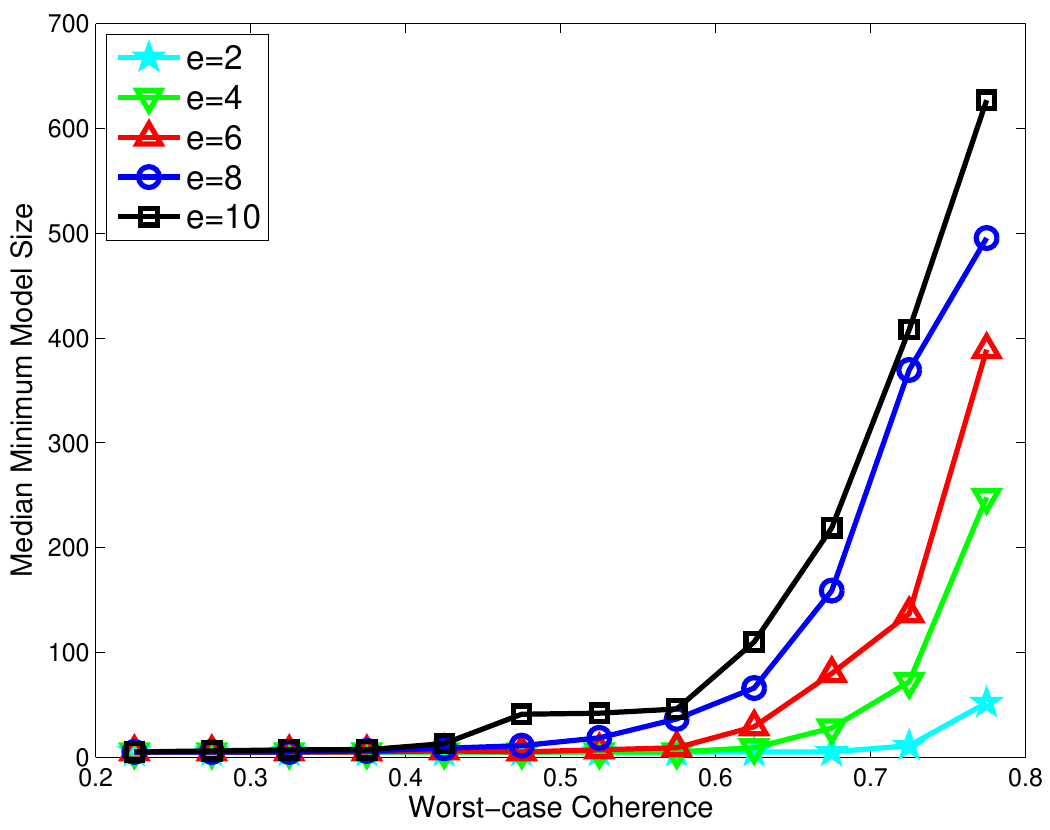} \label{fig:d_mu_oracle_median}}
\subfigure[]{ \includegraphics[scale=0.46]{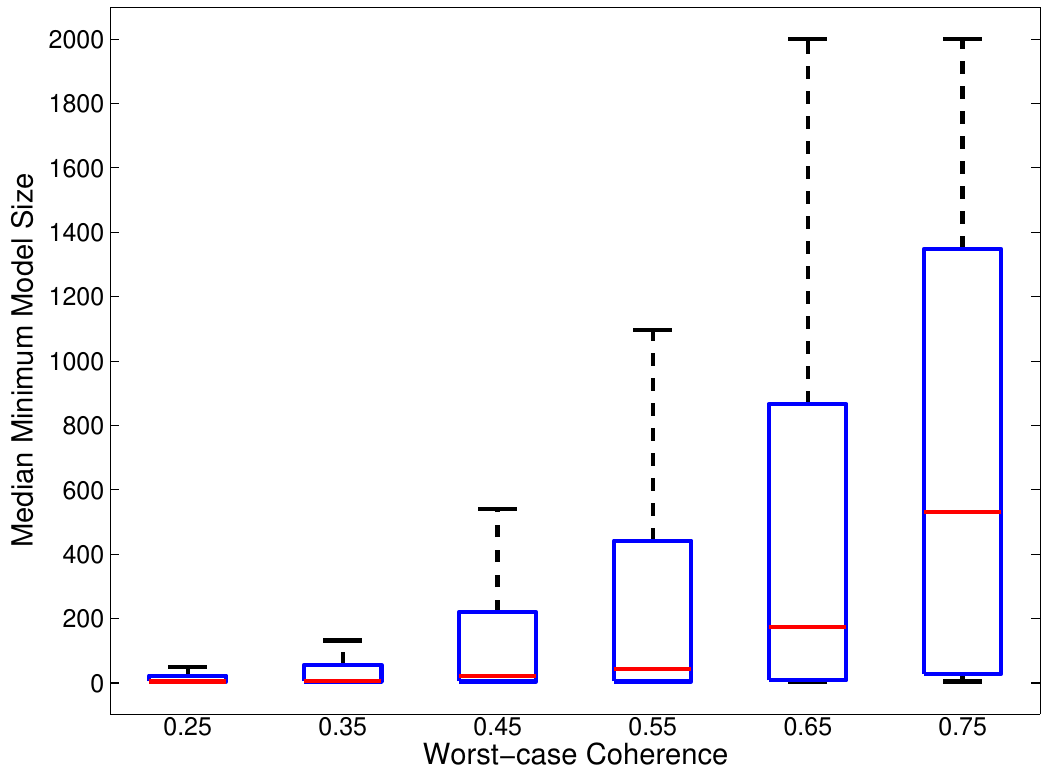} \label{fig:d_us_oracle_b_10_noOutliers}}
\caption{(a) For various values of worst-case coherence of the design matrix, the output of the oracle is computed. The relation between worst-case coherence and minimum model size is analyzed for various MSR values. (b) Boxplot of minimum model size for various values of worst-case coherence. The boxplot corresponds to the plot in Fig. \ref{fig:d_mu_oracle_median} for $e=10$. As the worst-case coherence becomes large, there are instances when oracle has to select all $2000$ predictors to ensure $\Strue \subset \Shat$ (see boxplot for $\mu = 0.65$ and $0.75$).}
\end{figure}

More interestingly, if we focus on the plot in Fig.~\ref{fig:d_mu_oracle_median} for $e=10$, and we plot the relationship between $\mu$ and median MMS along with the interquartile range of MMS for each value of $\mu$, it can be seen that there are instances when the oracle has to select all $2000$ predictors to ensure $\Strue \subset \Shat$ (see boxplot for $\mu = 0.65$ and $0.75$). In other words, no screening can be performed at all in these cases. This phenomenon is also reflected in Theorem \ref{th:main_result_gen}: when $b(n,p)$ becomes too large, the condition imposed on MSR is no longer true and our analysis cannot be used for screening using Algorithm~\ref{algo:correlation_screening}. Thus, this experiment shows that the condition imposed on MSR in Theorem \ref{th:main_result_gen} is not a mere artifact of our analysis: as $\mu$ (equivalently, $b(n,p)$) becomes large, one may not be able to perform screening at all while ensuring $\Strue \subset \Shat$.

\subsection{Comparison with Screening Procedures for LASSO-type Methods}
\label{sec:Results_Gaussian}
In this section, we use Gaussian data to compare the performance of ExSIS to that of screening procedures for LASSO-type methods. The design matrix $X\in\mathbb{R}^{n\times p}$ contains entries from a standard Gaussian distribution such that the pairwise correlation between the variables is $\rho$. In this experiment, we fix $n$ at $200$ while we consider two models with $p=2000$ and $p=5000$. For each value of $p$, we further consider two models with $\rho=0.0$ and $\rho=0.3$. Thus, in our experiments, we consider four setups with $(p,\rho)$ $=$ $(2000,0.0)$, $(2000,0.3)$, $(5000,0.0)$ and $(5000,0.3)$ to analyze the impact of dimensionality and pairwise correlation on performance of the screening procedures for LASSO-type methods in relation to ExSIS. For each of these four different setups, the model size is set at $|\Strue|=5$, and the locations of the non-zero coefficients in the parameter vector $\beta$ are chosen such that $\Strue$ is a uniformly at random subset of $[[p]]$. The values of the non-zero coefficients in the parameter vector $\beta$ are generated from $|z|+2$ where $z$ is distributed as a standard Gaussian random variable. Furthermore, the noise samples are generated from a standard Gaussian distribution, and the response vector $y$ is generated using \eqref{eq:sys}. Finally, the response vector $y$ and the columns of the design matrix $X$ are normalized to have unit norm. 

\begin{figure}[h!]
\centering
\subfigure[]{ \includegraphics[scale=0.3]{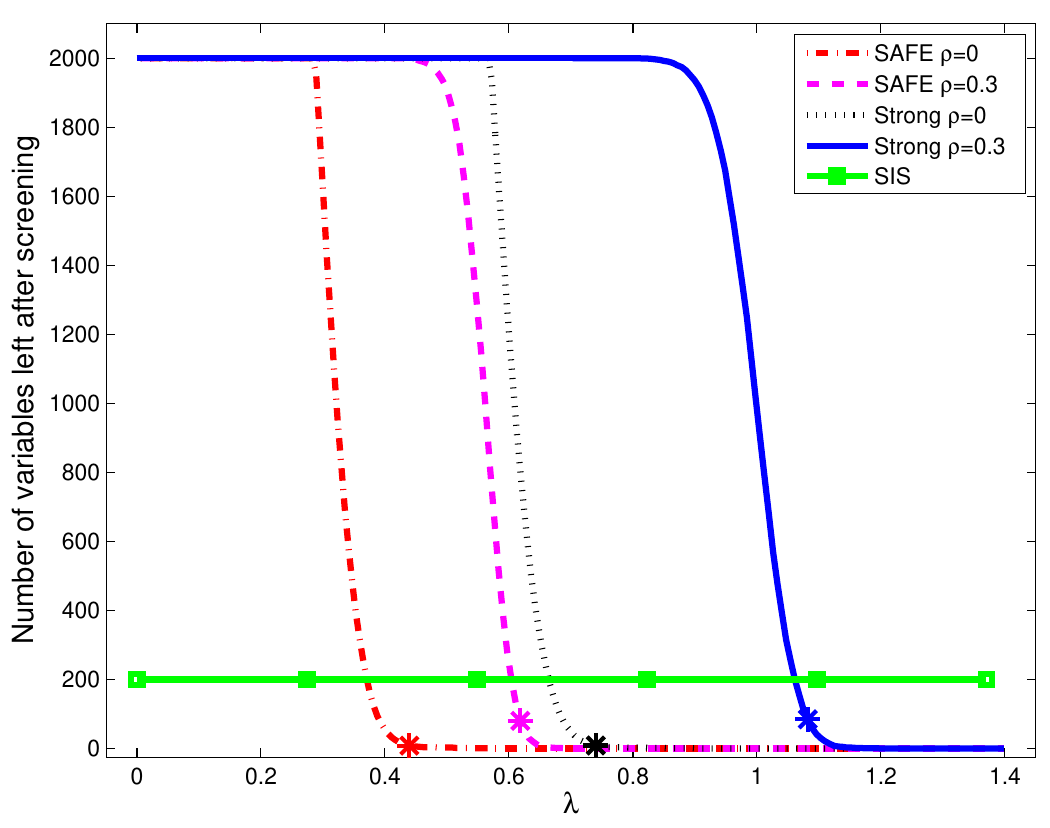} \label{fig:ScrMS_LASSO_p2000}}
\subfigure[]{ \includegraphics[scale=0.3]{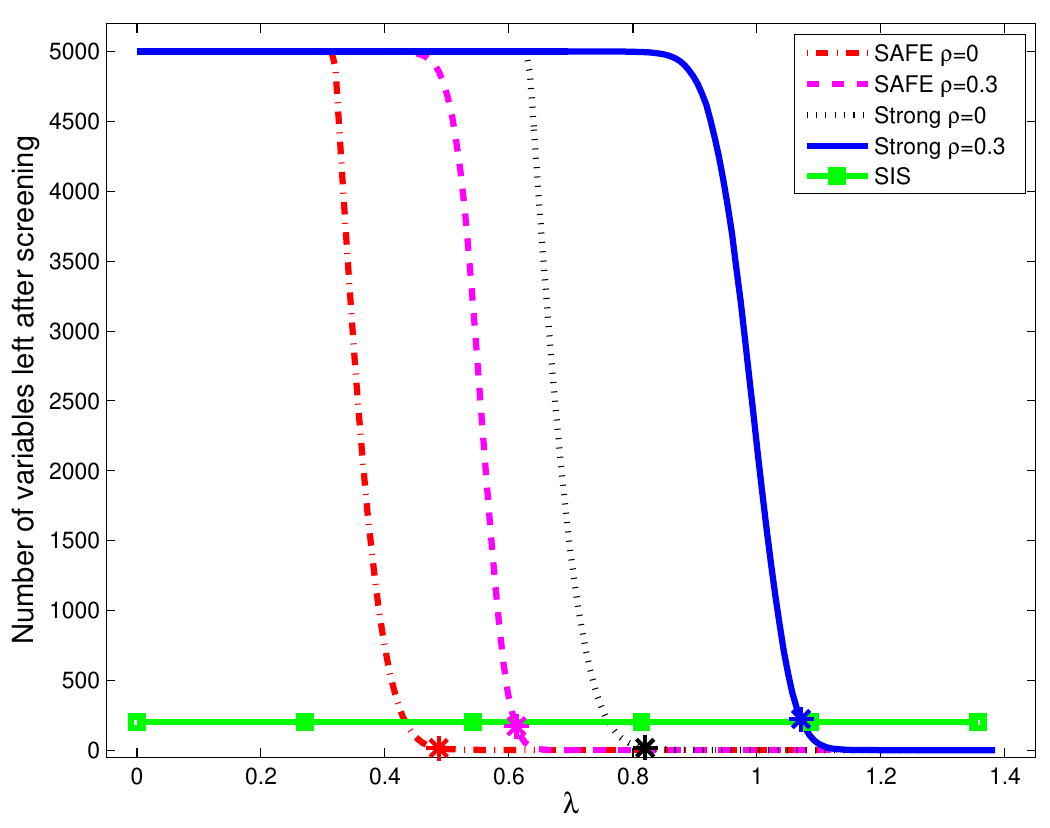} \label{fig:ScrMS_LASSO_p5000}}

\subfigure[]{ \includegraphics[scale=0.3]{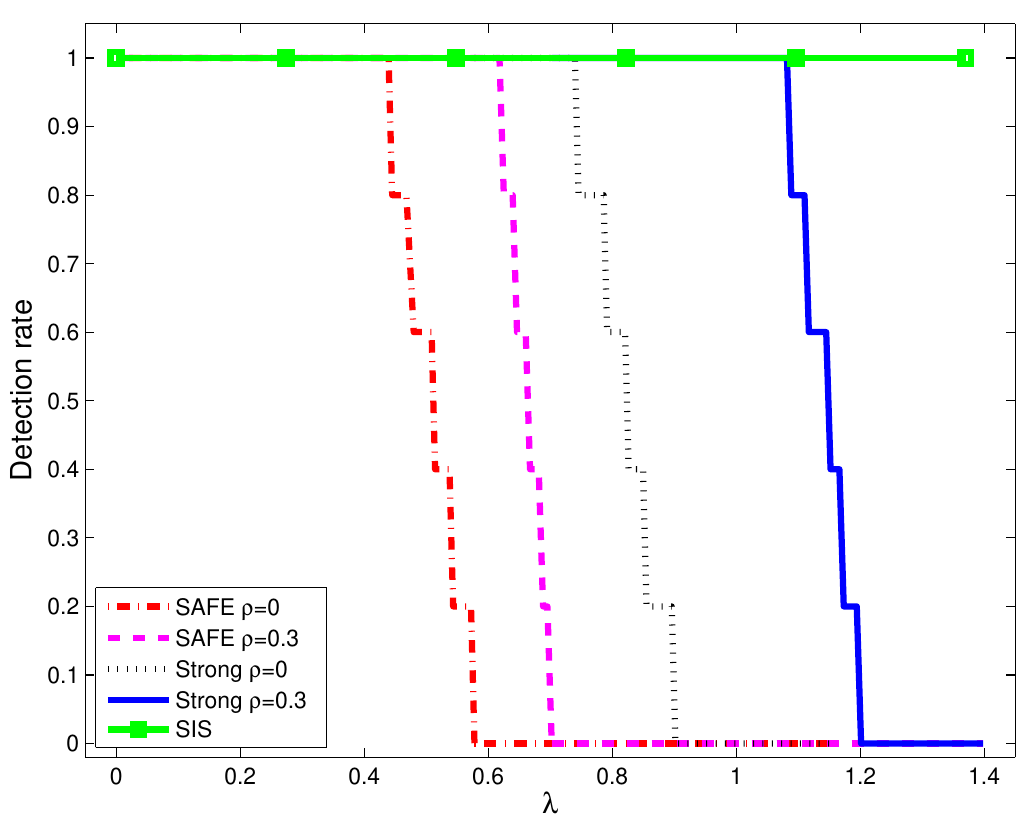} \label{fig:ScrVR_LASSO_p2000}}
\subfigure[]{ \includegraphics[scale=0.3]{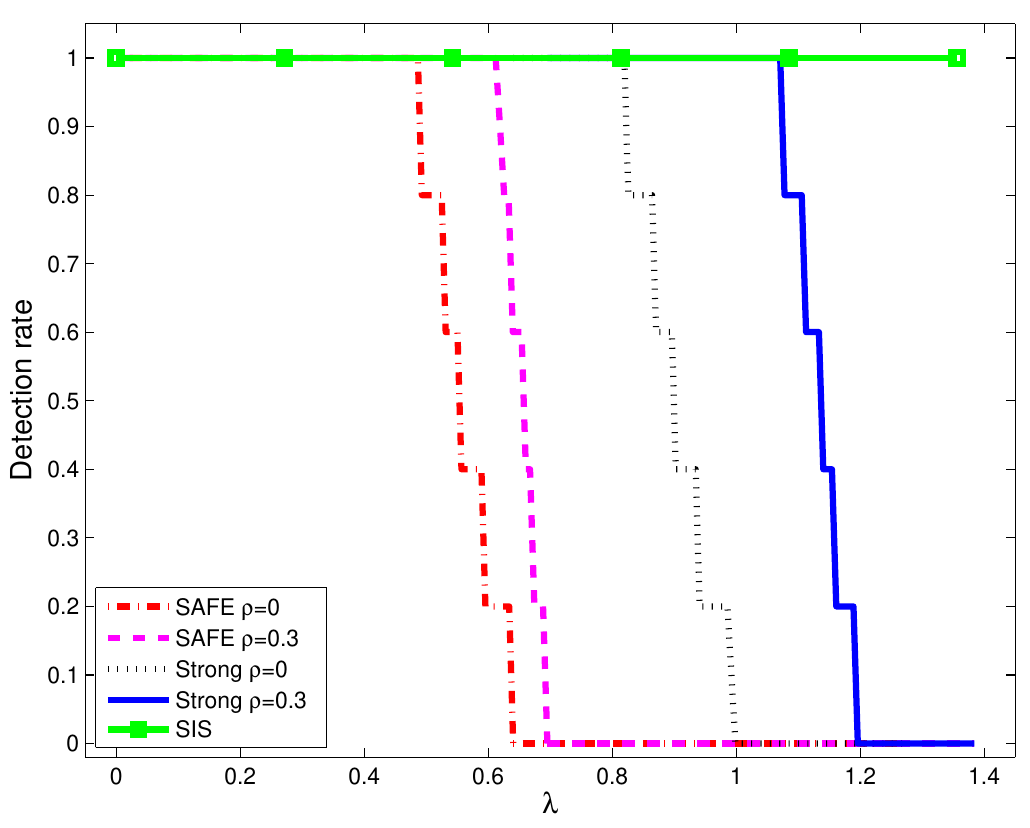} \label{fig:ScrVR_LASSO_p5000}}
\caption{Gaussian data matrices are screened using LASSO-based screening procedures for various model parameters. In (a) and (c), $p=2000$; whereas, in (b) and (d), $p=5000$. For each value of $p$, the screening experiment is repeated for $\rho=0.0$ and $\rho=0.3$. In each experiment, the model size after screening and the corresponding detection rate is evaluated for different values of the regularization parameter $\lambda$. The shown results are median over $100$ random draws of the data matrix $X$/parameter vector $\beta$/noise vector $\eta$ in \eqref{eq:sys}.}
\label{fig:Screening_LASSO}
\end{figure}

\begin{figure}[h!]
\centering
\subfigure[]{ \includegraphics[scale=0.3]{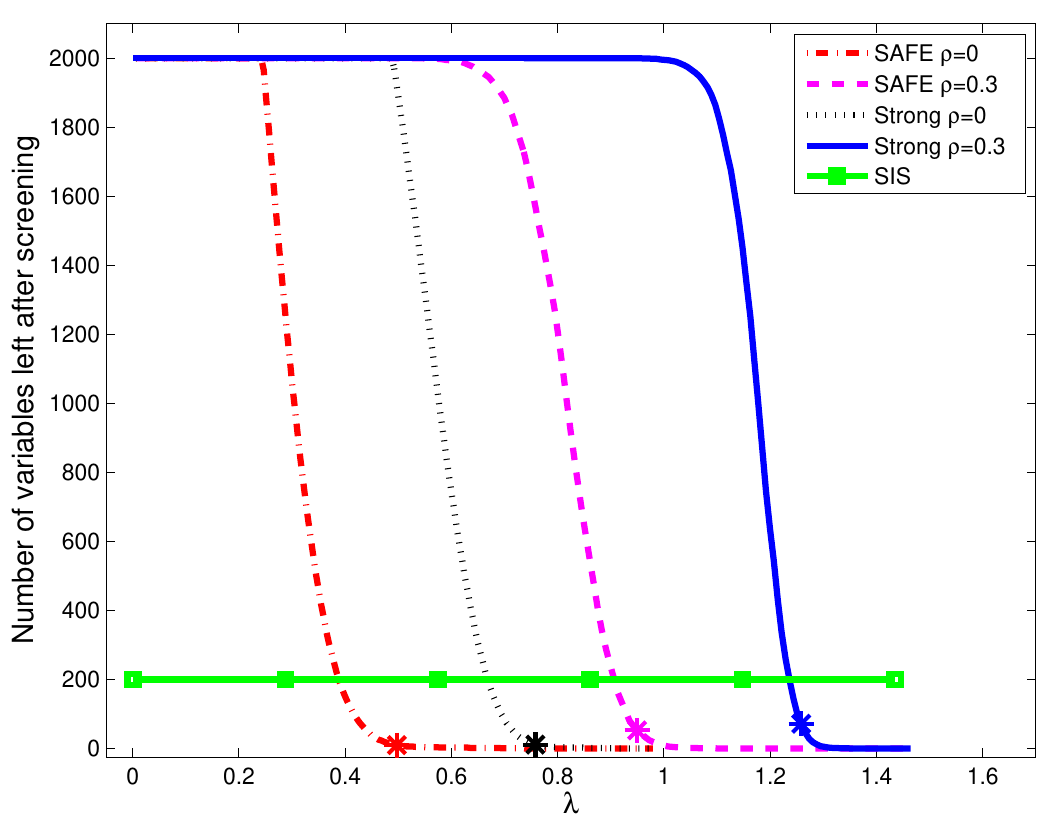} \label{fig:ScrMS_EN_p2000}}
\subfigure[]{ \includegraphics[scale=0.3]{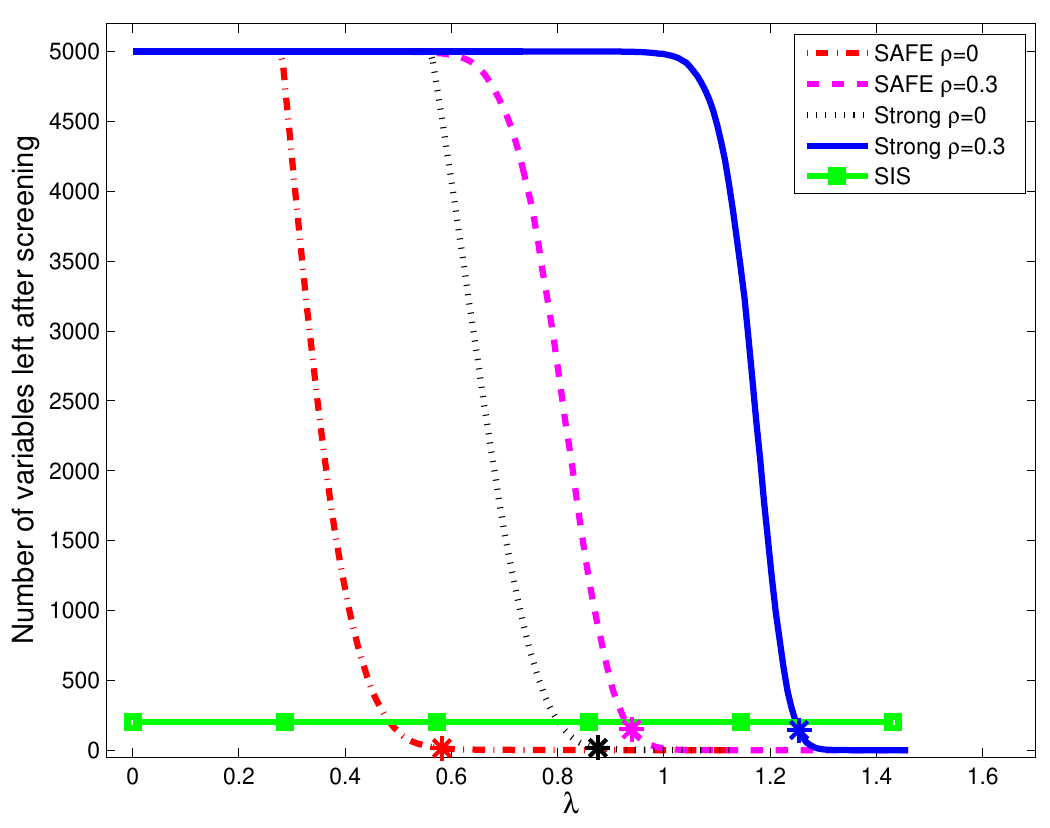} \label{fig:ScrMS_EN_p5000}}

\subfigure[]{ \includegraphics[scale=0.3]{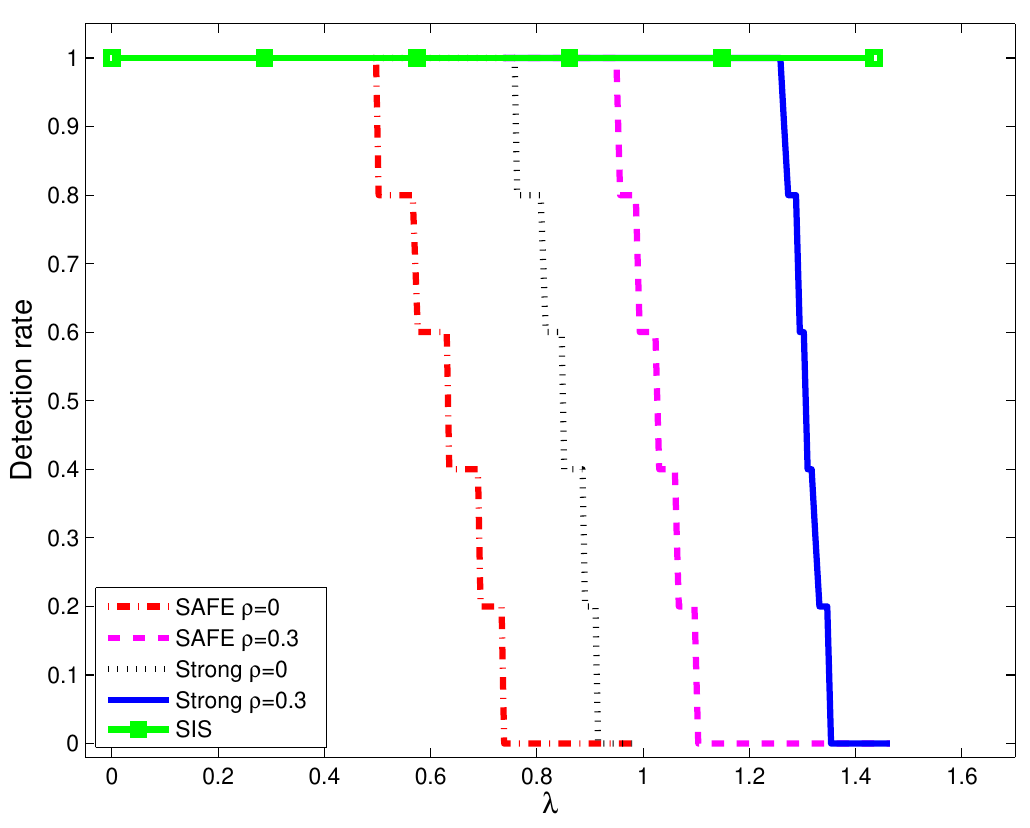} \label{fig:ScrVR_EN_p2000}}
\subfigure[]{ \includegraphics[scale=0.3]{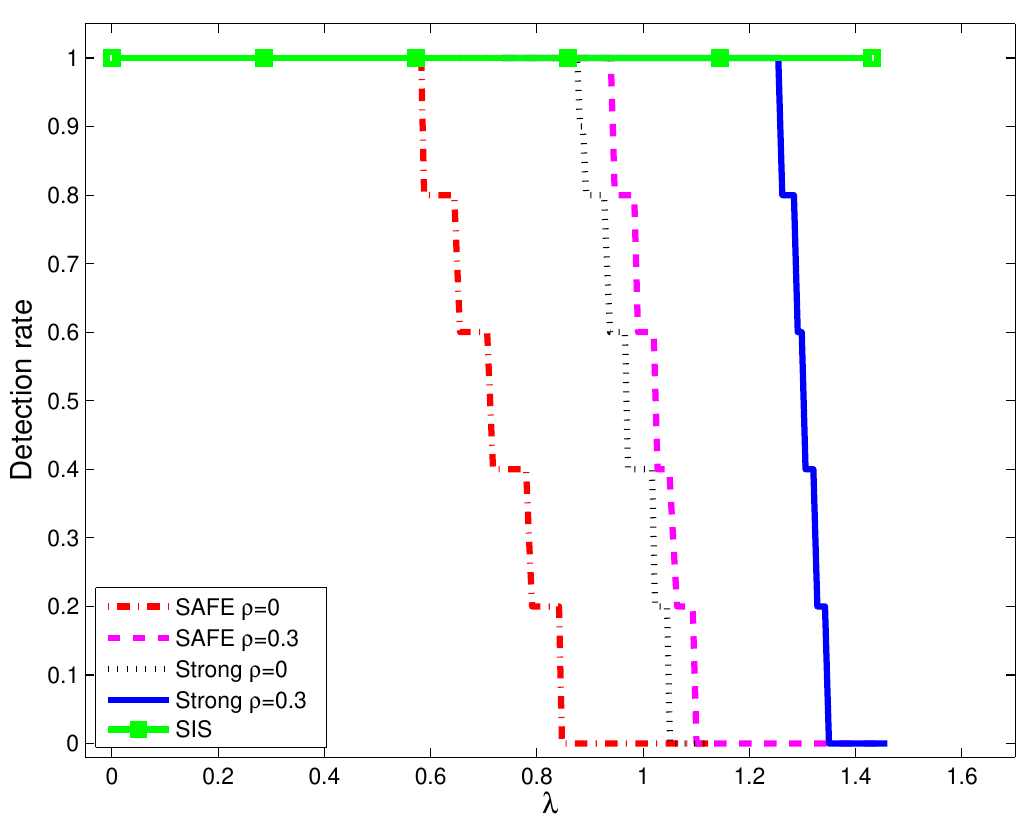} \label{fig:ScrVR_EN_p5000}}
\caption{Gaussian data matrices are screened using elastic net-based screening procedures for various model parameters. In (a) and (c), $p=2000$; whereas, in (b) and (d), $p=5000$. For each value of $p$, the screening experiment is repeated for $\rho=0.0$ and $\rho=0.3$. In each experiment, the model size after screening and the corresponding detection rate is evaluated for different values of the regularization parameter $\lambda$. The shown results are median over $100$ random draws of the data matrix $X$/parameter vector $\beta$/noise vector $\eta$ in \eqref{eq:sys}.}
\label{fig:Screening_EN}
\end{figure}

To analyze the performance of screening procedures for the LASSO method~\cite{tibshirani1996regression}, the columns of $X$ are screened using SAFE method~\cite{ghaoui2010safe} and strong rules~\cite{tibshirani2012strong} for LASSO. Recall that the LASSO problem can be expressed as
\begin{align*}
\hat{\beta} = \underset{\beta \in \mathbb{R}^p}{\operatorname{\arg\,min}} \frac{1}{2} \| y - X \beta\|_2^2 + \lambda \|\beta\|_1.
\end{align*}
For each of these screening methods, we perform screening of $X$ over a set of $200$ values of the regularization parameter $\lambda$ that are chosen uniformly from a linear scale. We compare the screening performance of SAFE method and strong rules for LASSO with the ExSIS method where $d=2n$. Note that our selection of the value of $d$ has some slack over the suggested value of $d$ from Corollary~\ref{corr:main_result_sub_Gaussian} because the conditions on $\log p$ and $\frac{\betamin}{\|\beta\|_2}$ in Corollary~\ref{corr:main_result_sub_Gaussian} don't hold true in this experiment.\footnote{In order for stated conditions to hold, we need significantly larger $n$ (and $p$); however, running LASSO on such large problems has high computational needs.} To compare the performance of these various screening methods, we use two metrics: ($i$) the model size (number of variables) after screening, which is defined as $|\Shat|$, and ($ii$) the detection rate, which is defined as $\frac{|\Strue \cap \Shat|}{|\Strue|}$. Using these metrics of performance, Fig.~\ref{fig:Screening_LASSO} shows the results of our simulations as median over $100$ draws of the random design matrix $X$/parameter vector $\beta$/noise vector $\eta$ in \eqref{eq:sys} for each of the four setups that we consider in this section. 

Next, the design matrix $X$ is also generated and screened using SAFE method~\cite{ghaoui2010safe} and strong rules~\cite{tibshirani2012strong} for elastic net~\cite{zou2005regularization} as explained before. Recall that the elastic net problem can be expressed as 
\begin{align*}
\hat{\beta} = \underset{\beta \in \mathbb{R}^p}{\operatorname{\arg\,min}} \frac{1}{2} \| y - X \beta\|_2^2 + \lambda_1 \|\beta\|_1 
+ \frac{1}{2} \lambda_2 \|\beta\|_2 .
\end{align*}
In our simulations, we use the parametrization $(\lambda_1, \lambda_2) = (\alpha \lambda, (1 - \alpha) \lambda)$ and we let $\alpha=0.5$. Fig.~\ref{fig:Screening_EN} shows the screening performance of SAFE method for elastic net, strong rules for elastic net and ExSIS over $100$ draws of the random design matrix $X$ for each of the four setups, as described before. In both Figs.~\ref{fig:Screening_LASSO} and \ref{fig:Screening_EN}, the largest value of $\lambda$ for which median detection rate is $1.0$ is labeled with an asterisk for each optimization-based screening procedure. In other words, for each screening procedure, only if $\lambda$ is smaller than the value of $\lambda$ labeled by an asterisk, the screening procedure maintains a median detection rate of $1.0$. Notice also that if the chosen value of $\lambda$ is too small, no variable is deleted by the screening procedure. Thus, as can be seen in both the figures, there is only a narrow range of values of $\lambda$ over which the optimization-based screening procedures are able to get rid of variables while maintaining a detection rate of $1.0$. Thus, from a practical point of view, it is not trivial to use SAFE method or strong rules for screening because there is no way of ensuring that the chosen value of $\lambda$ is within the narrow range of values of $\lambda$ for which significant screening can be performed while maintaining a detection rate of $1.0$. In comparison, the ExSIS method does not depend on the parameter $\lambda$, and in our experiments, it could always be used for screening while maintaining a median detection rate of $1.0$ (as shown in both Figs.~\ref{fig:Screening_LASSO} and \ref{fig:Screening_EN}). Before we end this discussion, note that, even within the narrow range of values of $\lambda$ for which SAFE method or strong rules can be used for screening while maintaining a detection rate of $1.0$, there is an even narrower range of values of $\lambda$ for which SAFE method or strong rules delete more variables than ExSIS.

\subsection{Sentiment Analysis of IMDb Movie Reviews and ExSIS}
\label{sec:Results_IMDb}

In high-dimensional classification, it has been shown that the presence of irrelevant variables increases the difficulty of classification, and the classification error tends to increase with dimensionality of the data model \cite{fan2008high, fan2014challenges}. Variable selection becomes important in high-dimensional classification as it can be used to discard the subset of irrelevant variables and reduce the dimensionality of the data model. Once the variable selection step is performed, classification can be performed based on the subset of relevant variables. In this section, we consider the problem of classifying IMDb movie reviews with positive or negative sentiments. In particular, we use variable selection to ($i$) reduce dimensionality of the data model, and ($ii$) learn a linear data model for classification (as explained later). To build and test our classification model, we make use of the IMDb movie reviews dataset~\cite{maas-EtAl:2011}, with the response being either a 1 (positive review) or a 0 (negative review), and we extract features using the \emph{term frequency-inverse document frequency method}~\cite{IRManning}. 

To increase the reliability of our results, the original dataset of 25K reviews is first randomly divided into five bins for five independent trials, with each bin further divided into 3K train and 2K test reviews. In each bin, before we use the 3K reviews for fitting a linear model on the feature space, we perform a preprocessing step to get rid of the features (words) that are highly correlated. Note that, when we refer to learning/fitting a linear model, we mean to estimate the vector $\beta$ in \eqref{eq:sys}. For learning the linear data model, we use LASSO as well as elastic net. For tuning the regularization parameter in LASSO as well as elastic net for each bin, we perform a five-fold cross validation experiment and choose the value of the regularization parameter that minimizes the mean square error on the training dataset. To evaluate the predictive power of the linear model, we use the notion of test \emph{true positive} (TP) rate, which is the percentage of the remaining 2K test movie reviews that are correctly classified by the model. For classification of the test reviews, we use the trained linear model to estimate the response for each test review. If the estimated response is less than 0.5 for a test review, the test review is assigned a negative sentiment and vice versa. The above procedure is repeated for each of the five bins of data and the average prediction accuracy is reported in Table \ref{table:IMDb}. The average model size before variable selection in the five runs of the experiment is $21,345$.

We also repeat the aforementioned experiment procedure but with a slight variation. For each of the five bins of data, we use Algorithm~\ref{algo:correlation_screening} to decrease dimensionality of the data model before performing the variable selection step. The objective of this variation in the experiment is to analyze the decrease in computational time and any change in the prediction accuracy when the variable selection step is preceded with a screening step using Algorithm~\ref{algo:correlation_screening}. To choose the value of $d$ in Algorithm~\ref{algo:correlation_screening}, we verify that the training data matrix for each fold of data satisfies the coherence property, and then we choose $d=2 n$ where $n=3000$. Note that the chosen value of $d$ has some slack over the suggested value of $d$ in Corollary~\ref{corr:mu_nu_conditions} because the condition on $\frac{\betamin}{\|\beta\|_2}$ in Corollary~\ref{corr:mu_nu_conditions} does not hold true in this experiment. After the screening step, we use LASSO and elastic net to learn and test a classification model as explained before. The results are reported in Table \ref{table:IMDb}.

\vspace{5 mm}
\begin{table}[h]
\caption{\toedit{Test} true positive (TP) rates and computational times for experiments on the IMDb dataset, averaged over the five folds of the IMDb dataset. The standard deviation is also reported in parenthesis.}
\label{table:IMDb}
\centering
\begin{tabular}{ |p{3cm}||p{3cm}|p{3cm}|  }
 \hline
 Training method  & Test TP rate &Training time\\
 \hline
 LASSO              & 83.01\;(0.77)	 & 388.35\;(26.66)\\
 ExSIS-LASSO        & 82.23\;(0.83)	 & 177.43\;(16.85)\\
 Elastic net 		& 84.35\;(0.89)  & 272.46\;(15.76) \\
 ExSIS-Elastic net  & 82.06\;(0.94)  & 111.20\;(4.65)  \\
 \hline
\end{tabular}
\end{table}
\vspace{5 mm}

Thus, we use LASSO and elastic net, both with and without screening, to train and test a linear model for classification of movie reviews. For each of these four cases, Table~\ref{table:IMDb} \todel{summarizes the train and test TP rates, which are the percentages of correctly classified reviews in train and test reviews, respectively}\toedit{reports the test TP rate, which is the percentage of correctly classified reviews within the test reviews}. The computational time needed for learning the linear model is also reported as an average over the five folds of data. It can be seen from the table that Algorithm~\ref{algo:correlation_screening} reduces the training time by a factor of more than two, while there is only a small decrease in predictive power of the trained model.

%% file: Conclusion.tex
\section{Conclusion}
\label{sec:Conclusion}
In this paper, we studied marginal correlation-based screening for ultrahigh-dimensional linear models. In our analysis, we provided sufficient conditions for sure screening of arbitrary linear models, which gave us novel insights on the screening procedure, while our simulation results verified that these insights are reflective of the challenges associated with marginal correlation-based screening. \toedit{We also provided polynomial-time verifiable conditions for arbitrary (random or deterministic) matrices under which the dimension of the model can be reduced to almost the sample size. Furthermore, we specialized our sufficient conditions for sub-Gaussian linear models and demonstrated that our analysis coincides with the existing results for screening of such models.} \todel{Furthermore, we evaluated our sufficient conditions for the family of sub-Gaussian matrices as well as arbitrary (random or deterministic) matrices, and we provided verifiable conditions under which the dimension of the model can be reduced to almost the sample size.} In our experiments with real-world data, we demonstrated the computational savings that can be achieved through ExSIS in high dimensional variable selection.

%% file: Appendix.tex
\begin{appendices}

\section{Proof of Lemma~\ref{lemma:frac0}}\label{app:proof_GeneralConditions.lemma.frac01}
We begin by defining
\begin{align}
w^{(i)} := \Xm_{\widehat{\mathcal{S}}_{p_{i-1}}}^{\top} \Ym = \Xm_{\widehat{\mathcal{S}}_{p_{i-1}}} ^{\top} \Xm_{\Strue} \beta_{\Strue} + \Xm_{\widehat{\mathcal{S}}_{p_{i-1}}} ^{\top} \eta =: \xi^{(i)} + \tilde{\eta}^{(i)}
\label{eq:corr2}
\end{align}
where $w^{(i)} \in \mathbb{R}^{|\widehat{\mathcal{S}}_{p_{i-1}}|}$ measures the correlation of the observation vector $y$ with each column of $\Xm_{\widehat{\mathcal{S}}_{p_{i-1}}}$. To derive an upper bound on $\numb_i$, we derive upper and lower bounds on $\sum \limits_{\ja=1}^{p_{i-1}} |w_{\ja}^{(i)}|$. A simple upper bound on $\sum \limits_{\ja=1}^{p_{i-1}} |w_{\ja}^{(i)}|$ is:
\begin{align}
\sum \limits_{\ja=1}^{p_{i-1}} |w_{\ja}^{(i)}| = \sum \limits_{\ja=1}^{p_{i-1}} |\xi_{\ja}^{(i)} + \tilde{\eta}_{\ja}^{(i)}| \leq \sum \limits_{\ja=1}^{p_{i-1}} (|\xi_{\ja}^{(i)}| + |\tilde{\eta}_{\ja}^{(i)}|)
= \|\xi^{(i)}\|_1 + \|\tilde{\eta}^{(i)}\|_1 \text{.}
\label{eq:num1}
\end{align}

Next, we recall that $w = X^{\top} y$ and we further define $\xi$ and $\tilde{\eta}$ such that
\begin{align}
w = \Xm^{\top} \Ym = \Xm^{\top} \Xm_{\Strue} \beta_{\Strue} + \Xm^{\top} \eta =: \xi + \tilde{\eta} \text{.}
\label{eq:corr2b}
\end{align}
Now define $\mathcal{T}_i:= \{\ja \in \widehat{\mathcal{S}}_{p_{i-1}}:|w_{\ja}| \geq \min \limits_{\jb \in \Strue} |w_{\jb}| \}$. Then, a simple lower bound on $\sum \limits_{\ja=1}^{p_{i-1}} |w_{\ja}^{(i)}|$ is:
\begin{align}
\sum \limits_{\ja=1}^{p_{i-1}} |w_{\ja}^{(i)}|
  &= \sum \limits_{\ja \in \widehat{\mathcal{S}}_{p_{i-1}}} |w_{\ja}|
  \geq \sum \limits_{\ja \in \mathcal{T}_i} |w_{\ja}|
\nonumber\\
  & \geq \sum \limits_{\ja \in \mathcal{T}_i} \min \limits_{\jb\in\mathcal{T}_i} |w_{\jb}|
  \geqa \sum \limits_{\ja \in \mathcal{T}_i} \min \limits_{\jb\in\Strue} |w_{\jb}|
\nonumber\\
  & = t_i (\min \limits_{\jb\in\Strue} |w_{\jb}|)
    = t_i  (\min \limits_{\jb\in\Strue} |\xi_{\jb} + \tilde{\eta}_{\jb}|)
\nonumber\\
 & \geq t_i  (\min \limits_{\jb\in\Strue} |\xi_{\jb}| - \max \limits_{\jb\in\Strue} |\tilde{\eta}_{\jb}|) \text{,}
\label{eq:den1}
\end{align}
where (a) follows from definition of $\mathcal{T}_i$. Combining \eqref{eq:num1} with \eqref{eq:den1}, we get
\begin{align}
\numb_i \leq \frac{\|\xi^{(i)}\|_1 + \|\tilde{\eta}^{(i)}\|_1}{ \min \limits_{\jb\in\Strue} |\xi_{\jb}| - \max \limits_{\jb\in\Strue} |\tilde{\eta}_{\jb}|} \text{.}
\label{eq:frac0}
\end{align}
We next bound $\|\xi^{(i)}\|_1$, $\max \limits_{\jb\in\Strue} |\tilde{\eta}_{\jb}|$, $\|\tilde{\eta}^{(i)}\|_1$ and $\min \limits_{\jb\in\Strue} |\xi_{\jb}|$ separately. First, we derive an upper bound on $\|\xi^{(i)}\|_1$:
\begin{align}
\|\xi^{(i)}\|_1
& = \sum \limits_{\ja \in \widehat{\mathcal{S}}_{p_{i-1}} } \big|\sum \limits_{\jb \in \Strue} X_{\ja}^{\top} \Xm_{\jb} \beta_{\jb} \big|
\nonumber\\
& \eqb \sum \limits_{\ja \in \Strue } \big|\sum \limits_{\jb \in \Strue} X_{\ja}^{\top} \Xm_{\jb} \beta_{\jb} \big|
 + \sum \limits_{\ja \in \widehat{\mathcal{S}}_{p_{i-1}} \setminus \Strue } \big|\sum \limits_{\jb \in \Strue} X_{\ja}^{\top} \Xm_{\jb} \beta_{\jb} \big|
\nonumber\\
&\leqc \sum \limits_{\ja \in \Strue } \big|\sum \limits_{\mathclap{\substack{\jb \in \Strue \\ \jb \neq \ja}}} X_{\ja}^{\top} \Xm_{\jb} \beta_{\jb} \big|
 + \sum \limits_{\ja\in\Strue} |\beta_{\ja}| + \sum \limits_{\ja \in \widehat{\mathcal{S}}_{p_{i-1}} \setminus \Strue } \big|\sum \limits_{\jb \in \Strue} X_{\ja}^{\top} \Xm_{\jb} \beta_{\jb} \big|
\nonumber\\
& \leq k \max \limits_{\ja \in \Strue} \big|\sum \limits_{\mathclap{\substack{\jb \in \Strue \\ \jb \neq \ja}}} X_{\ja}^{\top} \Xm_{\jb} \beta_{\jb} \big|
 + (p_{i-1}-k) \max \limits_{\ja \in \widehat{\mathcal{S}}_{p_{i-1}} \setminus \Strue} \big|\sum \limits_{\jb \in \Strue} X_{\ja}^{\top} \Xm_{\jb} \beta_{\jb} \big| + \|\beta\|_1
 \nonumber\\
& \leq k \max \limits_{\ja \in \Strue} \big|\sum \limits_{\mathclap{\substack{\jb \in \Strue \\ \jb \neq \ja}}} X_{\ja}^{\top} \Xm_{\jb} \beta_{\jb} \big|
 + (p_{i-1}-k) \max \limits_{\ja \in \Struec} \big|\sum \limits_{\jb \in \Strue} X_{\ja}^{\top} \Xm_{\jb} \beta_{\jb} \big| + \|\beta\|_1,
\label{eq:xi_up_main}
\end{align}
where (b) follows since $\Strue \subset \widehat{\mathcal{S}}_{p_{i-1}}$ and (c) follows from the triangle inequality and the fact that the columns of $X$ are unit norm. Next, we have
\begin{align}
\max \limits_{\jb\in\Strue} |\tilde{\eta}_{\jb}|
\leq \|\tilde{\eta}\|_{\infty} \leq 2 \sqrt{\sigma^2 \log p},
\label{eq:eta2}
\end{align}
where the last inequality follows from conditioning on $\Gn$. Similarly, we have
\begin{align}
\|\tilde{\eta}^{(i)}\|_1
&= \sum \limits_{\ja \in \widehat{\mathcal{S}}_{p_{i-1}}} |X_{\ja}^{\top} \eta|
\leq \sum \limits_{\ja \in \widehat{\mathcal{S}}_{p_{i-1}}} \max \limits_{\jb \in \widehat{\mathcal{S}}_{p_{i-1}}} |X_{\jb}^{\top} \eta|
\nonumber\\
&= p_{i-1} (\max \limits_{\jb \in \widehat{\mathcal{S}}_{p_{i-1}}} |X_{\jb}^{\top} \eta|)
\leq 2 p_{i-1} \sqrt{\sigma^2 \log p}
\label{eq:eta3}
\end{align}
where the last inequality, again, follows from $\Gn$. Last, we lower bound $\underset{\ja \in \Strue} \min |\xi_{\ja}|$ as follows:
\begin{align}
\underset{\ja \in \Strue} \min |\xi_{\ja}|
& = \underset{\ja \in \Strue} \min | \sum \limits_{\jb \in \Strue} X_{\ja}^{\top} \Xm_{\jb} \beta_{\jb} |
\nonumber\\
& \eqd \underset{\ja \in \Strue} \min | \sum \limits_{\mathclap{\substack{\jb \in \Strue \\ \jb \neq \ja}}}  X_{\ja}^{\top} \Xm_{\jb} \beta_{\jb} + \beta_{\ja}  |
\nonumber\\
& \geq  \underset{\ja \in \Strue} \min |\beta_{\ja}| - \underset{\ja \in \Strue} \max  | \sum \limits_{\mathclap{\substack{\jb \in \Strue \\ \jb \neq \ja}}} X_{\ja}^{\top} \Xm_{\jb} \beta_{\jb} |
= \betamin - \underset{\ja \in \Strue} \max  | \sum \limits_{\mathclap{\substack{\jb \in \Strue \\ \jb \neq \ja}}} X_{\ja}^{\top} \Xm_{\jb} \beta_{\jb}|
\label{eq:xi_down_main}
\end{align}
where (d) follows because the columns of $X$ are unit norm. Combining \eqref{eq:xi_up_main}, \eqref{eq:eta2}, \eqref{eq:eta3}, \eqref{eq:xi_down_main} with \eqref{eq:frac0}, we obtain
\begin{align}
\numb_i \leq
\frac{k \max \limits_{\ja \in \Strue} | \sum \limits_{\mathclap{\substack{\jb \in \Strue \\ \jb \neq \ja}}} X_{\ja}^{\top} \Xm_{\jb} \beta_{\jb} | + (p_{i-1} -k) \max \limits_{\ja \in \Struec} |\sum \limits_{\jb \in \Strue} X_{\ja}^{\top} \Xm_{\jb} \beta_{\jb} | + \|\beta\|_1 + 2 p_{i-1} \sqrt{\sigma^2 \log p}}
 {\betamin - \underset{\ja \in \Strue} \max  | \sum \limits_{\mathclap{\substack{\ja \in \Strue \\ \jb \neq \ja}}} X_{\ja}^{\top} \Xm_{\jb} \beta_{\jb} |  - 2 \sqrt{\sigma^2 \log p} } \text{.}
\label{eq:frac0_a}
\end{align}
Assuming the $(k, b)\mhyphen$screening condition for the matrix $X$ holds, we finally obtain
\begin{align}
\numb_i &\leq
\frac{k b(n,p) + (p_{i-1} -k) b(n,p) + \|\beta\|_1 + 2 p_{i-1} \sqrt{\sigma^2 \log p}}
 {\betamin - b(n,p) - 2 \sqrt{\sigma^2 \log p} }
 \nonumber\\
&= \frac{ p_{i-1} b(n,p) + \|\beta\|_1 + 2 p_{i-1} \sqrt{\sigma^2 \log p}}
 {\betamin - b(n,p) - 2 \sqrt{\sigma^2 \log p} } \text{.}
\label{eq:frac0_b}
\end{align}
This completes the proof of the lemma.\qed

\section{Proof of Lemma~\ref{lemma:gen_one_iter_stoc}}\label{app:proof_GeneralConditions.one.iter.stoc}
For $\bar{\numb}_i < p_{i-1}$, we need
\begin{align}
 & \frac{ p_{i-1} b(n,p) \|\beta\|_2 + \|\beta\|_1 + 2 p_{i-1} \sqrt{\sigma^2 \log p}}
 {\betamin - b(n,p) \|\beta\|_2  - 2 \sqrt{\sigma^2 \log p} } < p_{i-1}
 \nonumber\\
& \Leftrightarrow
p_{i-1} > \frac{ \frac{ \|\beta\|_1}{\|\beta\|_2}}{\frac{\betamin}{\|\beta\|_2} - 2 b(n,p) - \frac{4 \sqrt{\sigma^2 \log p}}{\|\beta\|_2}}.
\label{eq:lemma_gen_scr_pt_0}
\end{align}
Since $\|\beta\|_1 \leq \sqrt{k} \|\beta\|_2$, we have
\begin{align}
p_{i-1} > \frac{\sqrt{k}}{\frac{\betamin}{\|\beta\|_2} - 2 b(n,p) - \frac{4 \sqrt{\sigma^2 \log p}}{\|\beta\|_2}}
\label{eq:lemma_gen_scr_pt_3}
\end{align}
as a sufficient condition for \eqref{eq:lemma_gen_scr_pt_0}. Thus, \eqref{eq:lemma_gen_scr_pt_3} is a sufficient condition for $\bar{\numb}_i < p_{i-1}$.\qed

\section{Proof of Lemma~\ref{lemma:Det_mu}}\label{app:proof_ArbitraryMatrices.mu}
Notice that $\max \limits_{i\in \Strue} | \sum \limits_{\mathclap{\substack{j\in \Strue \\ j \neq i}}} X_i^{\top} \Xm_{j} \beta_j |
 \leq \max \limits_{i\in \Strue}  \sum \limits_{\mathclap{\substack{j\in \Strue \\ j \neq i}}} |X_i^{\top} \Xm_{j} \beta_j | \leq \max \limits_{i\in \Strue} \sum \limits_{\mathclap{\substack{j\in \Strue \\ j \neq i}}} |X_i^{\top} \Xm_{j}| |\beta_j |$. Further, we have
\begin{align}
\max \limits_{i\in \Strue} \sum \limits_{\mathclap{\substack{j\in \Strue \\ j \neq i}}} |X_i^{\top} \Xm_{j}| |\beta_j |
 \leq \max \limits_{i\in \Strue}  \sum \limits_{\mathclap{\substack{j\in \Strue \\ j \neq i}}} \mu |\beta_j | \leq \mu \|\beta\|_1 \leq \mu \sqrt{k} \|\beta\|_2 \text{.}
\end{align}
An identical argument also establishes that $\max \limits_{i\in \Struec} |\sum \limits_{j\in \Strue} X_i^{\top} \Xm_{j} \beta_j | \leq  \mu \|\beta\|_1 \leq \mu \sqrt{k} \|\beta \|_2 $.\qed

\section{Proof of Lemma~\ref{lemma:StOC_SC}}\label{app:proof_ArbitraryMatrices.mu.nu}

The proof of Lemma~\ref{lemma:StOC_SC} relies on the following two lemmas, which are formally proved in \cite{bajwa2010gabor}.
\begin{lemma}[\!\!\mbox{\cite[Lemma~3]{bajwa2010gabor}}]
Assume $\Strue$ is drawn uniformly at random from $k$-subsets of $[[p]]$ and choose a parameter $a \geq 1$. Let $\epsilon \in [0,1)$ and $k \leq \min \{ \epsilon^2 \nu^{-2}, (1+a)^{-1}p \}$. Then, with probability exceeding $1 - 4k \exp( - \frac{(\epsilon - \sqrt{k} \nu)^2} {16(2+a^{-1})^2 \mu^2} )$, we have
\begin{align*}
\max \limits_{i\in \Strue} | \sum \limits_{\mathclap{\substack{j\in \Strue \\ j \neq i}}} X_i^{\top} \Xm_{j} \beta_j |
\leq \epsilon \, \|\beta\|_2.
\end{align*}
\label{lemma:StOCa}
\end{lemma}
\begin{lemma}[\!\!\mbox{\cite[Lemma~4]{bajwa2010gabor}}]
Assume $\Strue$ is drawn uniformly at random from $k$-subsets of $[[p]]$ and choose a parameter $a \geq 1$. Let $\epsilon \in [0,1)$ and $k \leq \min \{ \epsilon^2 \nu^{-2}, (1+a)^{-1}p \}$. Then, with probability exceeding $1 - 4(p-k) \exp( - \frac{(\epsilon - \sqrt{k} \nu)^2} {8(2+a^{-1})^2 \mu^2} )$, we have
\begin{align*}
\max \limits_{i\in \Struec} | \sum \limits_{\mathclap{\substack{j\in \Strue}}} X_i^{\top} \Xm_{j} \beta_j |
\leq \epsilon \, \|\beta\|_2.
\end{align*}
\label{lemma:StOCb}
\end{lemma}
Using a simple union bound with Lemma \ref{lemma:StOCa} and Lemma \ref{lemma:StOCb}, we have
 \begin{align*}
&\max \limits_{i\in \Strue} | \sum \limits_{\mathclap{\substack{j\in \Strue \\ j \neq i}}} X_i^{\top} \Xm_{j} \beta_j | \leq \epsilon \, \|\beta\|_2 \text{, and}
 \\
&\max \limits_{i\in \Struec} | \sum \limits_{\mathclap{\substack{j\in \Strue}}} X_i^{\top} \Xm_{j} \beta_j | \leq \epsilon \, \|\beta\|_2
 \end{align*}
with probability exceeding $1 - \delta$ where $\delta = 4p \exp( - \frac{(\epsilon - \sqrt{k} \nu)^2} {16(2+a^{-1})^2 \mu^2} )$. Fix $\epsilon = c_{\mu} \mu \sqrt{\log p}$ and $a=9$. Then, the claim is that $\delta \leq 4 p^{-1}$. Next, we will prove our claim. Before we can fix $\epsilon = c_{\mu} \mu \sqrt{\log p}$ and $a=9$, we need to ensure that the chosen values of $a$, $\epsilon$ and the allowed values of $k$ satisfy the assumptions in Lemma \ref{lemma:StOCa} and Lemma \ref{lemma:StOCb}. First, note that $\epsilon<1$ because of $\mu < \frac{1}{c_{\mu }\sqrt{\log p}}$. Second, $k \leq \frac{p}{1+a}$ because of $a=9$, $p \geq \max \{2n, \exp(5) \}$ and $k\leq \frac{n}{\log p}$. Third, and last, $k \leq \frac{\epsilon^2}{(9 \nu)^2}$ because of $\nu < \frac{\mu}{\sqrt{n}}$, $k\leq \frac{n}{\log p}$, $p\geq 2n$ and $c_{\mu}> 10 \sqrt{2}$. Finally, we use $\sqrt{k} \nu \leq \frac{\epsilon}{9}$ with $a=9$, $c_{\mu}> 10 \sqrt{2}$ and $\epsilon = c_{\mu} \mu \sqrt{\log p}$ to obtain
\begin{align*}
\exp( - \frac{(\epsilon - \sqrt{k} \nu)^2} {16(2+a^{-1})^2 \mu^2} ) \leq p^{-2}
\end{align*}
\noindent and thus $\delta \leq 4 p^{-1}$.\qed

\section{Proof of Lemma~\ref{lemma:subGa}}\label{app:proof_RandomMatrices.subGa}
Since $X_i := \frac{V_i}{\|V_i\|_2}$, we can write
\begin{align*}
\max \limits_{i\in \Strue} \big|\sum \limits_{\mathclap{\substack{j\in \Strue \\ j \neq i}}} X_i^{\top} \Xm_{j} \beta_j\big|
= \max \limits_{i\in \Strue} \big| \sum \limits_{\mathclap{\substack{j\in \Strue \\ j \neq i}}} \frac{V_i^\top}{\|V_i\|_2} \frac{V_j}{\|V_j\|_2} \beta_j \big| \text{.}
\end{align*}
We next fix an $i' \in \Strue$ and derive a probabilistic bound on $|\sum \limits_{\mathclap{\substack{j\in \Strue \\ j \neq i'}}} X_{i'}^{\top} \Xm_{j} \beta_j|$. This involves deriving both upper and lower probabilistic bounds on $\sum \limits_{\mathclap{\substack{j\in \Strue \\ j \neq i'}}} X_{i'}^{\top} \Xm_{j} \beta_j$ below in Step~1 and Step~2, respectively.

\noindent \textbf{Step 1 (Upper Bound):}
To provide an upper probabilistic bound on $\sum \limits_{\mathclap{\substack{j\in \Strue \\ j \neq i'}}} X_{i'}^{\top} \Xm_{j} \beta_j$, we first establish that $\|V_j\|_2 \geq \sqrt{\frac{n \sigma_j^2}{2}}$ for each $j\in\Strue$ with high probability in Step~1a and then we derive an upper probabilistic bound on $\sum \limits_{\mathclap{\substack{j\in \Strue \\ j \neq i'}}} \frac{V_{i'}^\top}{\|V_{i'}\|_2} \frac{\sqrt{2} V_j}{\sqrt{n \sigma_j^2}} \beta_j$ in Step~1b. We then combine these two steps in Step~1c for the final upper probabilistic bound on $\sum \limits_{\mathclap{\substack{j\in \Strue \\ j \neq i'}}} X_{i'}^{\top} \Xm_{j} \beta_j$.

\textbf{Step 1a:}
Note that
\begin{align}
\Pr \bigg[ \|V_j\|_2 < \sqrt{\frac{n \sigma_j^2}{2}} \bigg] \leq \exp \Big(- \frac{n}{8} \Big( \frac{\sigma_j}{4 b_j} \Big)^4 \Big)
\label{eq:subGLower}
\end{align}
for any $j\in \Strue$ \cite[eq.~(2.20)]{wainwright2015high}. Next, let $\mathcal{G}_{u,a}$ be the event that $\|V_j\|_2 \geq \sqrt{\frac{n \sigma_j^2}{2}}$ for all $j\in\Strue\setminus \{i'\}$. Then
\begin{align}
\Pr[\mathcal{G}_{u,a}^ c]
&= \Pr \bigg[ \bigcup \limits_{\mathclap{\substack{j\in \Strue \\ j \neq i'}}} \Big\{ \|V_j\|_2 < \sqrt{\frac{n \sigma_j^2}{2}} \Big\}  \bigg]
\leq \sum \limits_{\mathclap{\substack{j\in \Strue \\ j \neq i'}}} \Pr \bigg[ \|V_j\|_2 < \sqrt{\frac{n \sigma_j^2}{2}} \bigg]
\leq \sum \limits_{\mathclap{\substack{j\in \Strue \\ j \neq i'}}} \exp \Big(- \frac{n}{8} \Big( \frac{\sigma_j}{4 b_j} \Big)^4 \Big)
\nonumber\\
&\leq \sum \limits_{\mathclap{\substack{j\in \Strue \\ j \neq i'}}} \exp \Big(- \frac{n}{8} \Big( \frac{\sigma_*}{4 b_*} \Big)^4 \Big)
= (k-1) \exp \Big(- \frac{n}{8} \Big( \frac{\sigma_*}{4 b_*} \Big)^4 \Big).
\label{eq:BoundGua}
\end{align}

\textbf{Step 1b:}
Define
\begin{align*}
Y_{i'}:= \sum \limits_{\mathclap{\substack{j\in \Strue \\ j \neq i'}}}   \sqrt{\frac{2}{n \sigma_j^2}} \frac{V_{i'}^\top}{\|V_{i'}\|_2} V_j \beta_j,
\end{align*}
and let $\mathcal{G}_{u,b}$ be the event that $Y_{i'} \leq \sqrt{\frac{8 \log p}{n}} \big(\frac{b_*}{\sigma_*} \big) \|\beta\|_2$.
Then, the claim is that $\Pr(\mathcal{G}_{u,b}) \geq 1 - \frac{1}{p^2}$. In order to prove this claim, let us define another event $\Gubar:= \{ V_i' = v_{i'}\}$. Then, defining $u_{i'} := \frac{v_{i'}}{\|v_{i'}\|_2}$, we have
\begin{align}
M_{Y_{i'}} (\lambda | \, \Gubar)
&:= \mathbb{E}[\exp(\lambda Y_{i'}) | \, \Gubar]
= \mathbb{E}\Big[ \exp \Big( \lambda \sum \limits_{\mathclap{\substack{j\in \Strue \\ j \neq i'}}}  \sqrt{\frac{2}{n \sigma_j^2}} \frac{V_{i'}^\top}{\|V_{i'}\|_2} V_j \beta_j \Big) \big| \, \Gubar \Big]
= \mathbb{E}\Big[ \exp \Big( \lambda \sum \limits_{\mathclap{\substack{j\in \Strue \\ j \neq i'}}}  \sqrt{\frac{2}{n \sigma_j^2}} u_{i'}^{\top} V_j \beta_j \Big) \Big]
\nonumber\\
&= \mathbb{E}\Big[ \prod \limits_{\mathclap{\substack{j\in \Strue \\ j \neq i'}}} \exp \Big( \lambda \sqrt{\frac{2}{n \sigma_j^2}} u_{i'}^{\top} V_j \beta_j \Big)  \Big]
= \mathbb{E}\Big[ \prod \limits_{\mathclap{\substack{j\in \Strue \\ j \neq i'}}} \exp \Big( \lambda \sqrt{\frac{2}{n \sigma_j^2}} \sum \limits_{l=1}^{n} u_{{i'},l} V_{j,l} \beta_j \Big)  \Big]
\nonumber\\
&= \mathbb{E}\Big[ \prod \limits_{\mathclap{\substack{j\in \Strue \\ j \neq i'}}} \prod \limits_{l=1}^{n} \exp \Big( \lambda \sqrt{\frac{2}{n \sigma_j^2}} u_{{i'},l} V_{j,l} \beta_j \Big) \Big]
= \prod \limits_{\mathclap{\substack{j\in \Strue \\ j \neq i'}}} \prod \limits_{l=1}^{n} \mathbb{E}\Big[\exp \Big( \lambda \sqrt{\frac{2}{n \sigma_j^2}} u_{{i'},l} V_{j,l} \beta_j \Big) \Big]
\nonumber\\
&\leq \prod \limits_{\mathclap{\substack{j\in \Strue \\ j \neq i'}}} \prod \limits_{l=1}^{n} \exp \Big( \frac{\lambda^2}{n \sigma_j^2} u_{{i'},l}^2 b_j^2 \beta_j^2 \Big)
= \prod \limits_{\mathclap{\substack{j\in \Strue \\ j \neq i'}}} \exp \Big(  \frac{\lambda^2}{n \sigma_j^2} b_j^2 \beta_j^2 \sum \limits_{l=1}^{n} u_{{i'},l}^2 \Big)
= \prod \limits_{\mathclap{\substack{j\in \Strue \\ j \neq i'}}} \exp \Big( \frac{\lambda^2}{n \sigma_j^2} b_j^2 \beta_j^2 \Big)
\nonumber\\
&= \exp \Big( \frac{\lambda^2}{n} \sum \limits_{\mathclap{\substack{j\in \Strue \\ j \neq i'}}}  \frac{b_j^2 \beta_j^2}{\sigma_j^2} \Big)
\leq \exp \Big( \frac{\lambda^2}{n} \Big(\frac{b_*}{\sigma_*}\Big)^2 \sum \limits_{\mathclap{\substack{j\in \Strue \\ j \neq i'}}} \beta_j^2 \Big)
\leq \exp \Big( \frac{\lambda^2}{n} \Big(\frac{b_*}{\sigma_*}\Big)^2 \|\beta\|_2^2 \Big) \text{.}
\end{align}
By the Chernoff bound on $Y_{i'}$, we next obtain
\begin{align}
\label{eqn:app.proof.chernoff}
\Pr(Y_{i'} > a | \, \Gubar)
&\leq \min \limits_{\lambda>0} \exp(-\lambda a) M_{Y_{i'}} (\lambda | \, \Gubar)
\nonumber\\
&\leq \min \limits_{\lambda>0} \exp(-\lambda a) \exp \Big( \frac{\lambda^2}{n} \Big(\frac{b_*}{\sigma_*}\Big)^2 \|\beta\|_2^2 \Big)
\nonumber\\
&= \exp \Big( -\frac{1}{4} \Big(\frac{\sigma_*}{b_*}\Big)^2 \frac{a^2 \, n}{\|\beta\|_2^2} \Big).
\end{align}
Substituting $a=\sqrt{\frac{8 \log p}{n}} \Big(\frac{b_*}{\sigma_*}\Big) \|\beta\|_2$ in \eqref{eqn:app.proof.chernoff}, we obtain
\begin{align}
\Pr\left(Y_{i'} > \sqrt{ \frac{8 \log p}{n}} \Big(\frac{b_*}{\sigma_*}\Big) \|\beta\|_2 \Big| \, \Gubar\right) \leq \frac{1}{p^2}.
\end{align}
Finally, by the law of total probability, we obtain
\begin{align}
\Pr\left(Y_{i'} > \sqrt{ \frac{8 \log p}{n}} \Big(\frac{b_*}{\sigma_*}\Big) \|\beta\|_2\right) &= \mathbb{E}_{V_i'}\left[\Pr\left(Y_{i'} > \sqrt{ \frac{8 \log p}{n}} \Big(\frac{b_*}{\sigma_*}\Big) \|\beta\|_2 \Big| \, \Gubar\right)\right]\nonumber\\
&\leq \mathbb{E}_{V_i'}\left[\frac{1}{p^2}\right] = \frac{1}{p^2}.
\label{eq:BoundGub}
\end{align}
Thus, the event $\mathcal{G}_{u,b}$ holds with probability exceeding $1 - \frac{1}{p^2}$.

\textbf{Step 1c:}
Conditioning on $\mathcal{G}_{u,a} \cap \mathcal{G}_{u,b}$, we have from \eqref{eq:BoundGua} and \eqref{eq:BoundGub} that
\begin{align}
\sum \limits_{\mathclap{\substack{j\in \Strue \\ j \neq i'}}} X_{i'}^{\top} \Xm_{j} \beta_j
\leq \sqrt{ \frac{8 \log p}{n}} \Big(\frac{b_*}{\sigma_*}\Big) \|\beta\|_2.
\label{eq:BoundGu}
\end{align}
Further, note that
\begin{align*}
\Pr( \mathcal{G}_{u,a} \cap \mathcal{G}_{u,b} )
&\geq 1 - \Pr(\mathcal{G}_{u,a}^c) - \Pr(\mathcal{G}_{u,b}^c)
\nonumber\\
&\geq 1 - (k-1) \exp \Big(- \frac{n}{8} \Big( \frac{\sigma_*}{4 b_*} \Big)^4 \Big) - \frac{1}{p^2}
\nonumber\\
&\geqa 1 - \frac{(k-1)}{p^2} - \frac{1}{p^2} = 1 - \frac{k}{p^2},
\end{align*}
where (a) follows since $\log p \leq \frac{n}{16} \big( \frac{\sigma_*}{4 b_*} \big)^4$. Thus, \eqref{eq:BoundGu} holds with probability exceeding $1 - \frac{k}{p^2}$.

\noindent \textbf{Step 2 (Lower Bound):} Our claim in this step is that $\sum \limits_{\mathclap{\substack{j\in \Strue \\ j \neq i'}}} X_{i'}^{\top} \Xm_{j} \beta_j \geq - \sqrt{ \frac{8 \log p}{n}} \Big(\frac{b_*}{\sigma_*}\Big) \|\beta\|_2$ with probability exceeding $1 - \frac{k}{p^2}$. To establish this claim, notice that
\begin{align}
  \sum \limits_{\mathclap{\substack{j\in \Strue \\ j \neq i'}}} X_{i'}^{\top} \Xm_{j} \beta_j \geq - \sqrt{ \frac{8 \log p}{n}} \Big(\frac{b_*}{\sigma_*}\Big) \|\beta\|_2 \quad \Longleftrightarrow \quad - \sum \limits_{\mathclap{\substack{j\in \Strue \\ j \neq i'}}} X_{i'}^{\top} \Xm_{j} \beta_j \leq \sqrt{ \frac{8 \log p}{n}} \Big(\frac{b_*}{\sigma_*}\Big) \|\beta\|_2.
\end{align}
Further, we have $- \sum \limits_{\mathclap{\substack{j\in \Strue \\ j \neq i'}}} X_{i'}^{\top} \Xm_{j} \beta_j \equiv - \sum \limits_{\mathclap{\substack{j\in \Strue \\ j \neq i'}}} \frac{V_{i'}^\top}{\|V_{i'}\|_2} \frac{V_j}{\|V_j\|_2} \beta_j = \sum \limits_{\mathclap{\substack{j\in \Strue \\ j \neq i'}}} \frac{V_{i'}^\top}{\|V_{i'}\|_2} \frac{\widetilde{V}_j}{\|V_j\|_2} \beta_j$, where $\widetilde{V}_j := -V_j$ is still distributed as $V_j$ because of the symmetry of sub-Gaussian distributions. The claim now follows from a repetition of the analysis carried out in Step~1.

\noindent \textbf{Final Step:}
Step~1 and Step~2, along with the union bound, imply that
\begin{align*}
\big| \sum \limits_{\mathclap{\substack{j\in \Strue \\ j \neq i}}} X_{i'}^{\top} \Xm_{j} \beta_j \big|
\leq \sqrt{\frac{8 \log p}{n}} \Big(\frac{b_*}{\sigma_*} \Big) \|\beta\|_2
\end{align*}
with probability exceeding $1 - \frac{2k}{p^2}$. Next, notice that
\begin{align}
&\Pr \bigg[\max \limits_{i\in \Strue} \big| \sum \limits_{\mathclap{\substack{j\in \Strue \\ j \neq i}}} X_i^{\top} \Xm_{j} \beta_j \big|
\leq \sqrt{\frac{8 \log p}{n}} \Big(\frac{b_*}{\sigma_*} \Big) \|\beta\|_2 \, \bigg]
\nonumber\\
&= 1 -  \Pr \bigg[ \bigcup \limits_{i \in \Strue} \Big\{ \big| \sum \limits_{\mathclap{\substack{j\in \Strue \\ j \neq i}}} X_i^{\top} \Xm_{j} \beta_j \big| > \sqrt{\frac{8 \log p}{n}} \Big(\frac{b_*}{\sigma_*} \Big) \|\beta\|_2 \, \Big\}  \bigg]
\nonumber\\
&\geq 1 -  \sum \limits_{i \in \Strue} \Pr \bigg[ \big| \sum \limits_{\mathclap{\substack{j\in \Strue \\ j \neq i}}} X_i^{\top} \Xm_{j} \beta_j \big| > \sqrt{\frac{8 \log p}{n}} \Big(\frac{b_*}{\sigma_*} \Big) \|\beta\|_2 \, \bigg]
\nonumber\\
&\geq 1 -  \sum \limits_{i \in \Strue} \frac{2k}{p^2} = 1 -  \frac{2k^2}{p^2}.
\label{eq:BoundFinala}
\end{align}
This complete the proof of the lemma.\qed

\section{Proof of Lemma~\ref{lemma:subGb}}\label{app:proof_RandomMatrices.subGb}
Once again, notice that
\begin{align*}
\max \limits_{i\in \Strue^c} | \sum \limits_{\mathclap{\substack{j\in \Strue}}} X_i^{\top} \Xm_{j} \beta_j |
= \max \limits_{i\in \Strue^c} \Big| \sum \limits_{\mathclap{\substack{j\in \Strue}}} \frac{V_i^\top}{\|V_i\|_2} \frac{V_j}{\|V_j\|_2} \beta_j \Big| \text{.}
\end{align*}
We next fix an $i' \in \Struec$. Similar to the proof of Lemma~\ref{lemma:subGa}, the plan is to first derive a probabilistic bound on $|\sum \limits_{\mathclap{\substack{j\in \Strue}}} X_{i'}^{\top} \Xm_{j} \beta_j|$ and then use the union bound to provide a probabilistic bound on $\max \limits_{i\in \Struec} | \sum \limits_{\mathclap{\substack{j\in \Strue}}} X_i^{\top} \Xm_{j} \beta_j |$. Using steps similar to the ones in the proof of Lemma~\ref{lemma:subGa}, it is straightforward to establish that
\begin{align*}
\big| \sum \limits_{\mathclap{\substack{j\in \Strue}}} X_{i'}^{\top} \Xm_{j} \beta_j \big|
\leq \sqrt{\frac{8 \log p}{n}} \Big(\frac{b_*}{\sigma_*} \Big) \|\beta\|_2
\end{align*}
with probability exceeding $1 - \frac{2(k+1)}{p^2}$. The union bound finally gives us
\begin{align}
&\Pr \bigg[\max \limits_{i\in \Struec} \big| \sum \limits_{\mathclap{\substack{j\in \Strue}}} X_i^{\top} \Xm_{j} \beta_j \big|
\leq \sqrt{\frac{8 \log p}{n}} \Big(\frac{b_*}{\sigma_*} \Big) \|\beta\|_2 \, \bigg]
\nonumber\\
&= 1 -  \Pr \bigg[ \bigcup \limits_{i \in \Struec} \Big\{ \big| \sum \limits_{\mathclap{\substack{j\in \Strue}}} X_i^{\top} \Xm_{j} \beta_j \big| > \sqrt{\frac{8 \log p}{n}} \Big(\frac{b_*}{\sigma_*} \Big) \|\beta\|_2 \, \Big\}  \bigg]
\nonumber\\
&\geq 1 -  \sum \limits_{i \in \Struec} \Pr \bigg[ \big| \sum \limits_{\mathclap{\substack{j\in \Strue}}} X_i^{\top} \Xm_{j} \beta_j \big| > \sqrt{\frac{8 \log p}{n}} \Big(\frac{b_*}{\sigma_*} \Big) \|\beta\|_2 \, \bigg]
\nonumber\\
&\geq 1 -  \sum \limits_{i \in \Struec} \frac{2(k+1)}{p^2} = 1 -  \frac{2(k+1)(p-k)}{p^2}.
\label{eq:BoundFinala.2}
\end{align}
This completes the proof of the lemma.\qed

\end{appendices}